\documentclass{amsart}

\usepackage{amssymb}
\usepackage{mathtools}
\usepackage{booktabs}
\usepackage{array}
\usepackage{enumitem}
\usepackage{graphicx}
\usepackage[hidelinks]{hyperref}
\graphicspath{{figures/}{outputs/figures/}}

\newtheorem{theorem}{Theorem}[section]
\newtheorem{lemma}[theorem]{Lemma}
\newtheorem{proposition}[theorem]{Proposition}
\newtheorem{corollary}[theorem]{Corollary}
\theoremstyle{definition}

\theoremstyle{remark}

\numberwithin{equation}{section}
\newcommand{\K}{\mathbb K}
\newcommand{\C}{\mathbb C}
\newcommand{\Q}{\mathbb Q}
\newcommand{\Z}{\mathbb Z}
\newcommand{\N}{\mathbb N}
\newcommand{\id}{\operatorname{id}}
\newcommand{\End}{\operatorname{End}}
\newcommand{\Hom}{\operatorname{Hom}}
\newcommand{\Tr}{\operatorname{Tr}}
\newcommand{\Span}{\operatorname{span}}
\newcommand{\wt}{\operatorname{wt}}

\newcommand{\ev}{\operatorname{ev}}
\newcommand{\coev}{\operatorname{coev}}
\newcommand{\qtr}{\operatorname{qtr}}
\newcommand{\coqtr}{\operatorname{coqtr}}
\newcommand{\htd}{\operatorname{ht}_{d}}
\newcommand{\undercurvearrowleft}{%
  \raisebox{.5em}{\rotatebox{180}{$\curvearrowright$}}}
\newcommand{\undercurvearrowright}{%
  \raisebox{.5em}{\rotatebox{180}{$\curvearrowleft$}}}

\begin{document}

\title[Two-Parameter Quantum Link Invariants II]
{The Knot Invariant Associated to Two-Parameter Quantum Algebras II}

\author{Fan Zhaobing}
\author{Zhu Tianhao}

\begin{abstract}
% P0.01
Fan, Ma, and Xing constructed oriented-tangle invariants from finite-type
two-parameter quantum algebras.
In this paper, we construct an explicit parameter-transport comparison between
two-parameter modules and their corresponding one-parameter modules, and we
verify this comparison on every elementary oriented-tangle operator.
After extending scalars to a common coefficient field, we prove that if $M_1$
is any finite-dimensional integrable type-$1$ simple highest-weight
$U_{v,1}$-module whose weights lie in an admissible lattice and
$M_t=\Phi_t(M_1)$ is its transported module, then for every oriented link
$\mathcal L$, the corresponding normalized invariants satisfy
\[
 I_{v,t}^{M_t}(\mathcal L)=I_{v,1}^{M_1}(\mathcal L).
\]

% P0.02
Consequently, the two invariants assign equal values to exactly the same
pairs of oriented links and therefore have the same distinguishing power;
this includes the vector representations of the finite classical types
$A$, $B$, $C$, and $D$.
Sean Clark's comparison of ordinary and super quantum knot invariants is
obtained as a specialization of the same transport principle in which
explicit scalar factors are allowed.
\end{abstract}


\maketitle
\enlargethispage{3pt}
\setcounter{tocdepth}{2}
\tableofcontents

%==============================================================================
\section{Introduction}
%==============================================================================

% P1.01
Artin formalized the braid group, whose elements are isotopy classes of
finitely many strands running monotonically between two parallel planes
\cite{artin1925}.
The closure of a braid is obtained by joining corresponding upper and lower
endpoints outside the braid.
Alexander proved that every oriented link is such a closure
\cite{alexander1923}.
Markov identified the two elementary moves that preserve the isotopy class
of a braid closure \cite{markov1935}.
Jones later obtained his polynomial invariant from representations of braid
groups arising from operator algebras \cite{jones1985}.

% P1.02
Drinfeld and Jimbo introduced one-parameter quantum groups, which are
deformations, depending on one variable, of universal enveloping algebras of
semisimple Lie algebras \cite{drinfeld1987,jimbo1985}.
A universal $R$-matrix is an invertible tensor that interchanges the two
orders of the coproduct and satisfies the Yang--Baxter equation; on tensor
products of modules it therefore defines braid operators
\cite{kassel1995,lusztig1993}.
Here the coproduct is the algebra map that determines how an algebra element
acts on a tensor product of two modules.
The Yang--Baxter equation says that the two three-step ways of moving three
adjacent strands past one another give the same operator.
Reshetikhin and Turaev combined these operators with duality maps, which
algebraically bend an oriented strand, and twist maps, which record a chosen
transverse direction along it, to obtain invariants of colored framed tangles
\cite{reshetikhin-turaev1990,turaev1994}.
In this historical statement, colored means that a representation is assigned
to each component, while a framing is a continuously chosen transverse
direction along the component.
Turaev also gave an operator presentation of ordinary oriented-tangle
invariants in terms of crossings and duality morphisms \cite{turaev1990}.

% P1.03
Two-parameter and multiparameter quantum groups developed alongside the
one-parameter theory.
Takeuchi gave an early two-parameter quantization of $\mathrm{GL}(n)$
\cite{takeuchi1990}.
Benkart and Witherspoon studied two-parameter Drinfeld doubles
\cite{benkart-witherspoon2004}, as well as their representations and
Schur--Weyl duality \cite{benkart-witherspoon-rep2004}.
Fan and Li constructed $U_{v,t}$ from Cartan data and a nonsymmetric Ringel
form \cite{fan-li2015}.
Here Cartan data consist of simple-root indices and a symmetric bilinear form
encoding a generalized Cartan matrix, while a Ringel form is a bilinear form
whose symmetrization is that Cartan form.
Here $v$ and $t$ are independent formal parameters, and setting $t=1$ gives
the one-parameter specialization used in this paper.
Fan and Xing developed the corresponding deformed-double theory
\cite{fan-xing2019}.
Fan, Ma, and Xing then constructed its quasi-$R$-matrix, the part of the
braiding that changes weights, its checked $R$-matrix, which also includes
the flip of tensor factors, and the associated oriented-tangle functor
\cite{fan-ma-xing2024}.

% P1.04
For a positive integer $n$, $\mathfrak{osp}(1|2n)$ denotes the
orthosymplectic Lie superalgebra, and $\mathfrak{so}(2n+1)$ denotes the Lie
algebra of the odd-dimensional special orthogonal group.
Sean Clark constructed colored knot invariants from quantum covering groups
associated with $\mathfrak{osp}(1|2n)$ and compared them with invariants
associated with $\mathfrak{so}(2n+1)$ \cite{clark2017}.
A Lie superalgebra is a vector space graded by the two parity classes
$\Z/2\Z=\{\overline0,\overline1\}$, called even and odd, equipped with a
bracket whose skew-symmetry and Jacobi identity include the signs determined
by that grading.
A quantum covering group contains an auxiliary parameter whose two
specializations encode the ordinary and super cases \cite{clark2017}.
In addition, we prove that Sean Clark's comparison of ordinary and super
quantum knot invariants is a quantum-covering specialization of the general
scalar-defect transport theorem.

% P1.05
Our main result concerns finite Cartan type, meaning that the associated
Dynkin diagram, the graph that records the interactions among the simple
roots, is of finite type, and every finite-dimensional integrable type-$1$
simple highest-weight module $M_1$ whose weights lie in an admissible lattice,
together with its transported module $M_t=\Phi_t(M_1)$.
An admissible lattice, defined precisely in \eqref{eq:lattice}, is a lattice
containing the root lattice with finite index.
The symbol $\Phi_t$ denotes the explicit change of module action defined in
Section~\ref{sec:transport}; it does not denote an algebra isomorphism.
Here the index \(i\) ranges over the simple roots; \(K_i,K_i'\) are the
commuting invertible generators called toral generators, and \(E_i,F_i\) are
the corresponding raising and lowering generators.
A weight vector is a common eigenvector for all \(K_i,K_i'\).
The term type-$1$ means that these eigenvalues are exactly those in
\eqref{eq:weight-K}--\eqref{eq:weight-Kprime}; integrable means that, for
every vector, sufficiently high powers of each \(E_i\) and \(F_i\) annihilate
that vector; and simple highest-weight means that the module has no nonzero
proper submodule and is generated by a nonzero vector annihilated by every
\(E_i\).
Our main result states that
\begin{equation}\label{eq:intro-main}
 \boxed{I_{v,t}^{M_t}(\mathcal L)=I_{v,1}^{M_1}(\mathcal L)}
\end{equation}
for every oriented link $\mathcal L$.
Here $I_{v,t}^{M_t}(\mathcal L)$ is the scalar assigned to
$\mathcal L$ by the normalized two-parameter tangle functor colored by
$M_t$, and $I_{v,1}^{M_1}(\mathcal L)$ is the corresponding scalar after
the specialization $t=1$.
Consequently, for any two oriented links $\mathcal L_1,\mathcal L_2$,
\[
 I_{v,t}^{M_t}(\mathcal L_1)=I_{v,t}^{M_t}(\mathcal L_2)
 \quad\Longleftrightarrow\quad
 I_{v,1}^{M_1}(\mathcal L_1)=I_{v,1}^{M_1}(\mathcal L_2).
\]
This equivalence is the precise meaning of saying that the two invariants
have the same distinguishing power, or the same fineness.
The result applies uniformly to the classical vector representations of
types $A$, $B$, $C$, and $D$; affine Cartan data are not considered here.
The letters $A$, $B$, $C$, and $D$ label the four classical families of
finite Dynkin diagrams.

% P1.06
The proof has three steps.
First, we define the module transport $\Phi_t$ and an invertible
weight-diagonal map $J$ that compares the tensor product of transported
modules with the transport of their tensor product; such a comparison map is
called a tensorator.
Second, we prove that $J$ conjugates the checked $R$-matrices and is
compatible with all four evaluation and coevaluation maps that represent
oriented cups and caps.
Third, we present a link as a braid closure and show that the two boundary
copies of $J$ cancel inside the resulting closure trace, namely the ordinary
matrix trace with the weight-diagonal closure operator defined in
\eqref{eq:closure-trace}.
This proves equality of the complete tangle evaluations, rather than merely
equality of selected entries of a crossing matrix.

% P1.07
The paper has four sections.
Section~2 recalls the finite-type algebra, its weight modules, the
skew-Hopf pairing, the Fan--Ma--Xing quasi-$R$-matrix, and the ordinary
oriented-tangle construction of the link invariant.
Section~3 first defines parameter transport as a specific transfer of algebraic
structure and then proves the equality and records its consequences, including
Sean Clark's theorem as a specialization of the general transport statement.
Section~4 states the conclusions and the exact scope of the result.

%==============================================================================
\section{Preliminaries on Finite-Type Two-Parameter Quantum Algebras and Link Invariants}
\label{sec:preliminaries}
%==============================================================================

\subsection{Cartan data, lattices, and coefficient fields}

% P2.01
Throughout the paper, $\Z$ is the ring of integers, $\Q$ is the field of
rational numbers, and $\N=\{0,1,2,\ldots\}$.
The notation $\Z_{\leq0}$ denotes the set of nonpositive integers.
Let $I$ be a nonempty finite index set.
Let $\Omega=(\Omega_{ij})_{i,j\in I}$ be an integral matrix such that
\begin{equation}\label{eq:omega}
 \Omega_{ii}>0,\qquad \Omega_{ij}\leq0\ (i\ne j),\qquad
 \frac{\Omega_{ij}+\Omega_{ji}}{\Omega_{ii}}\in\Z_{\leq0}
 \ (i\ne j),
 \qquad \gcd_{i\in I}\Omega_{ii}=1.
\end{equation}
Here integral means that every matrix entry is an integer, and $\gcd$ denotes
the greatest common divisor.
Write $Q=\Z[I]$ for the free abelian group with basis $I$; it is called the
root lattice, and its basis element $i$ is also written $\alpha_i$ and called
the $i$th simple root.
A free abelian group with basis $I$ consists of the finite integer linear
combinations of elements of $I$, with unique coefficients.
The symbol $Q$ for this lattice must not be confused with the blackboard-bold
symbol $\Q$ for the rational-number field.
Define two bilinear forms on $Q$ by
\begin{equation}\label{eq:forms}
 \langle i,j\rangle=\Omega_{ij},
 \qquad i\mathbin{\cdot}j=\Omega_{ij}+\Omega_{ji}.
\end{equation}
The first form is the Ringel form, and the second is its symmetric part.
A bilinear form is additive in each argument, and symmetric means that its
value is unchanged when the two arguments are interchanged.
Define the positive integers $d_i$ and the Cartan integers $a_{ij}$ by
\begin{equation}\label{eq:cartan}
 d_i=\frac{i\mathbin{\cdot}i}{2}=\Omega_{ii},
 \qquad
 a_{ij}=\frac{i\mathbin{\cdot}j}{d_i}
       =\frac{2i\mathbin{\cdot}j}{i\mathbin{\cdot}i}.
\end{equation}
The matrix $A=(a_{ij})_{i,j\in I}$ is the symmetrizable generalized Cartan
matrix determined by $\Omega$; symmetrizable means that
$d_i a_{ij}=d_j a_{ji}$.
A generalized Cartan matrix is an integral matrix with diagonal entries
$a_{ii}=2$, off-diagonal entries $a_{ij}\leq0$, and
$a_{ij}=0$ exactly when $a_{ji}=0$; these properties follow here from
\eqref{eq:omega}.
We assume that $A$ is of finite type, equivalently that its Dynkin diagram
is a finite-type Dynkin diagram or that the symmetric form in
\eqref{eq:forms} is positive definite on $Q\otimes_{\Z}\Q$.
Positive definite means that $\lambda\mathbin{\cdot}\lambda>0$ for every
nonzero $\lambda$ in this rational vector space.
The pair $(I,\mathbin{\cdot})$ is then a finite Cartan datum, and the
antisymmetric part of the Ringel form is the source of the parameter $t$.

% P2.02
An admissible lattice is a free abelian group $L$ satisfying
\begin{equation}\label{eq:lattice}
 Q\subseteq L\subseteq Q\otimes_{\Z}\Q,
 \qquad [L:Q]<\infty.
\end{equation}
Here $[L:Q]$ is the number of cosets of $Q$ in $L$.
The tensor product $Q\otimes_{\Z}\Q$ is the rational vector space spanned by
the simple roots, so every $\lambda\in L$ has unique rational coordinates
$\lambda=\sum_{i\in I}\lambda_i i$.
Both bilinear forms in \eqref{eq:forms} extend uniquely to $L$ by
$\Q$-bilinearity.
Define the weighted height of $\lambda\in L$ by
\begin{equation}\label{eq:weighted-height}
 \htd(\lambda)=\sum_{i\in I}d_i\lambda_i.
\end{equation}
Choose a positive integer $D$ such that
\begin{equation}\label{eq:denominators}
 \lambda\mathbin{\cdot}\mu,\quad
 \langle\lambda,\mu\rangle,\quad
 \htd(\lambda)
 \in D^{-1}\Z
 \qquad(\lambda,\mu\in L),
\end{equation}
where $D^{-1}\Z=\{a/D:a\in\Z\}$.
Fix algebraically independent variables $v$ and $t$, meaning that no
nonzero polynomial with rational coefficients vanishes at the pair
$(v,t)$.
We work over
\begin{equation}\label{eq:field}
 \K_D=\Q(v^{1/D},t^{1/D}).
\end{equation}
Thus $\K_D$ is the field of rational functions in the displayed fractional
powers, and $v^r$ and $t^r$ are literal elements of $\K_D$ for every
$r\in D^{-1}\Z$.
We also write $v_\lambda=v^{\htd(\lambda)}$.

\subsection{The algebra \texorpdfstring{$U_{v,t}$}{U(v,t)}}

% P2.03
For the variables $v$ and $t$ fixed above, put
\begin{equation}\label{eq:local-parameters}
 v_i=v^{d_i},\qquad t_i=t^{d_i},
\end{equation}
and, for every $p\in\N$, define the two-parameter quantum integer and
factorial by
\begin{equation}\label{eq:q-integers}
 [p]_{v_i,t_i}
 =\frac{(v_it_i)^p-(v_it_i^{-1})^{-p}}
        {v_it_i-(v_it_i^{-1})^{-1}},
 \qquad
 [p]_{v_i,t_i}!=\prod_{r=1}^{p}[r]_{v_i,t_i}.
\end{equation}
The empty product is $[0]_{v_i,t_i}!=1$.
The one-parameter symbols $[p]_{v_i}$ and $[p]_{v_i}!$ are obtained by
setting $t=1$.
Direct simplification gives
\begin{equation}\label{eq:factorial}
 [p]_{v_i,t_i}=t_i^{p-1}[p]_{v_i},
 \qquad
 [p]_{v_i,t_i}!=t_i^{p(p-1)/2}[p]_{v_i}!.
\end{equation}

% P2.04
The two-parameter quantum algebra $U_{v,t}=U_{v,t}(I,\cdot,\Omega)$ is the
$\K_D$-algebra generated by the elements below and subject to the displayed
relations \cite[Section~2]{fan-ma-xing2024}:
\begin{equation}\label{eq:generators}
 E_i,F_i,K_i^{\pm1},(K_i')^{\pm1}\qquad(i\in I),
\end{equation}
with commuting invertible toral generators and relations
\begin{align}
 K_iE_jK_i^{-1}
 &=v^{i\cdot j}t^{\langle j,i\rangle-\langle i,j\rangle}E_j,
 &K_i'E_j(K_i')^{-1}
 &=v^{-i\cdot j}t^{\langle j,i\rangle-\langle i,j\rangle}E_j,
 \label{eq:torus-E}\\
 K_iF_jK_i^{-1}
 &=v^{-i\cdot j}t^{\langle i,j\rangle-\langle j,i\rangle}F_j,
 &K_i'F_j(K_i')^{-1}
 &=v^{i\cdot j}t^{\langle i,j\rangle-\langle j,i\rangle}F_j.
 \label{eq:torus-F}
\end{align}
The mixed relation is
\begin{equation}\label{eq:commutator}
 E_iF_j-F_jE_i
 =\delta_{ij}\frac{K_i-K_i'}{v_i-v_i^{-1}}.
\end{equation}
Here $\delta_{ij}$ is $1$ for $i=j$ and $0$ otherwise.

% P2.05
For $i\ne j$, define the positive integer $N_{ij}=1-a_{ij}$ and the two
Ringel-form entries $A_{ij}$ and $B_{ij}$ by
\begin{equation}\label{eq:NAB}
 N_{ij}=1-\frac{i\mathbin{\cdot}j}{d_i},
 \qquad A_{ij}=\langle i,j\rangle,
 \qquad B_{ij}=\langle j,i\rangle.
\end{equation}
The divided powers used below are
\begin{equation}\label{eq:divided}
 E_i^{(p)}=\frac{E_i^p}{[p]_{v_i,t_i}!},
 \qquad
 F_i^{(p)}=\frac{F_i^p}{[p]_{v_i,t_i}!}.
\end{equation}
For $p,p'\in\N$ with $p+p'=N_{ij}$, put
\begin{equation}\label{eq:kappa}
 \kappa_{ij}(p,p')=(-1)^p
 t^{-d_i pp'+pA_{ij}-pB_{ij}},
 \qquad p+p'=N_{ij},
\end{equation}
the Serre relations are
\begin{align}
 \sum_{p+p'=N_{ij}}\kappa_{ij}(p,p')E_i^{(p)}E_jE_i^{(p')}&=0,
 \label{eq:serre-E}\\
 \sum_{p+p'=N_{ij}}\kappa_{ij}(p,p')F_i^{(p')}F_jF_i^{(p)}&=0.
 \label{eq:serre-F}
\end{align}
The name Serre relation refers to these higher-order relations between the
root generators attached to two distinct simple roots.

% P2.06
A Hopf algebra is an algebra equipped with a coproduct $\Delta$, a counit
$\varepsilon$, and an antipode $S$.
The coproduct and counit are algebra homomorphisms, the coproduct is
coassociative, and the defining identities are
\begin{equation}\label{eq:Hopf-axioms}
 \begin{gathered}
 (\Delta\otimes\id)\Delta=(\id\otimes\Delta)\Delta,\qquad
 (\varepsilon\otimes\id)\Delta=\id=(\id\otimes\varepsilon)\Delta,\\
 m(S\otimes\id)\Delta=\eta\varepsilon
 =m(\id\otimes S)\Delta.
 \end{gathered}
\end{equation}
Here $m$ is multiplication, $\eta$ sends a scalar to that scalar times the
identity element, and $\id$ denotes the identity map; the usual identifications
of a vector space with its tensor product with the coefficient field are
understood in the counit identities.
For $U_{v,t}$, these three structure maps are determined on generators by
\begin{align}
 \Delta(K_i)&=K_i\otimes K_i,
 &\Delta(K_i')&=K_i'\otimes K_i',
 \label{eq:coproduct-torus}\\
 \Delta(E_i)&=E_i\otimes1+K_i\otimes E_i,
 &\Delta(F_i)&=1\otimes F_i+F_i\otimes K_i',
 \label{eq:coproduct-roots}\\
 S(E_i)&=-K_i^{-1}E_i,
 &S(F_i)&=-F_i(K_i')^{-1}.
 \label{eq:antipode-roots}
\end{align}
On the toral generators and for the counit,
\begin{equation}\label{eq:antipode-torus}
 S(K_i)=K_i^{-1},\quad S(K_i')=(K_i')^{-1},\qquad
 \varepsilon(K_i)=\varepsilon(K_i')=1,\quad
 \varepsilon(E_i)=\varepsilon(F_i)=0.
\end{equation}
The notation $U_{v,1}$ means that the coefficients in the defining
relations are evaluated at $t=1$, followed by extension of scalars from
$\Q(v^{1/D})$ to $\K_D$.

\subsection{Weight modules and dual modules}

% P2.08
An $L$-weight module is a finite-dimensional module
\begin{equation}\label{eq:weight-decomp}
 M=\bigoplus_{\lambda\in L}M_\lambda
\end{equation}
with finitely many nonzero summands, such that
\begin{equation}\label{eq:weight-shifts}
 E_iM_\lambda\subseteq M_{\lambda+i},
 \qquad F_iM_\lambda\subseteq M_{\lambda-i},
\end{equation}
and, for $m_\lambda\in M_\lambda$,
\begin{align}
 K_im_\lambda
 &=v^{i\cdot\lambda}
 t^{\langle\lambda,i\rangle-\langle i,\lambda\rangle}m_\lambda,
 \label{eq:weight-K}\\
 K_i'm_\lambda
 &=v^{-i\cdot\lambda}
 t^{\langle\lambda,i\rangle-\langle i,\lambda\rangle}m_\lambda.
 \label{eq:weight-Kprime}
\end{align}
Here finite-dimensional means that $M$ has a finite basis, each
$M_\lambda$ is a subspace called a weight space, and the direct-sum symbol
means that every vector of $M$ has a unique expression as a finite sum of
vectors from distinct weight spaces.
A vector is called homogeneous, or a weight vector, when it belongs to one
weight space $M_\lambda$.
The weight of a nonzero homogeneous vector $m_\lambda$ is written
$\wt(m_\lambda)=\lambda$.
The module is type-$1$ when the toral actions are exactly
\eqref{eq:weight-K}--\eqref{eq:weight-Kprime}.
It is integrable when every $E_i$ and $F_i$ acts locally nilpotently.
Explicitly, local nilpotence means that for every vector $m$ and every
$i\in I$, some positive power of $E_i$ and some positive power of $F_i$
annihilate $m$.
It is a simple highest-weight module of highest weight $\Lambda$ when it is
generated by a nonzero vector $m_\Lambda\in M_\Lambda$ satisfying
$E_i m_\Lambda=0$ for every $i$ and it has no nonzero proper submodule.

% P2.09
For later use, recall that the algebraic dual of a finite-dimensional module
$M$ is the vector space $M^*=\Hom_{\K_D}(M,\K_D)$ with action
\begin{equation}\label{eq:dual-action}
 (u\varphi)(m)=\varphi(S(u)m)
 \qquad(u\in U_{v,t},\ \varphi\in M^*,\ m\in M).
\end{equation}
The antipode in this formula is the map in
\eqref{eq:antipode-roots}--\eqref{eq:antipode-torus}; it is precisely what
makes evaluation a module homomorphism.
The notation $\Hom_{\K_D}(M,\K_D)$ means the vector space of all
$\K_D$-linear maps from $M$ to $\K_D$.
Its weight-$(-\lambda)$ subspace is
\begin{equation}\label{eq:dual-weight-space}
 (M^*)_{-\lambda}
 =\{\varphi\in M^*: \varphi(M_\mu)=0\text{ whenever }\mu\ne\lambda\}.
\end{equation}
Thus a homogeneous functional that is nonzero only on the weight space
$M_\lambda$ has weight $-\lambda$.

\subsection{The skew-Hopf pairing and the
\texorpdfstring{quasi-$R$-matrix}{quasi-R-matrix}}

% P2.10
Let $U_{v,t}^{+}$ be the subalgebra generated by all $E_i$, and let
$U_{v,t}^{-}$ be the subalgebra generated by all $F_i$.
The positive Borel half $U_{v,t}^{\geq0}$ is generated by
$E_i,K_i^{\pm1}$, and the negative Borel half $U_{v,t}^{\leq0}$ is
generated by $F_i,(K_i')^{\pm1}$.
Thus each Borel half consists of the toral generators together with the root
generators of one sign.
Give the root generators the degrees $\deg(E_i)=i$ and $\deg(F_i)=-i$.
Then
\[
 U_{v,t}^{+}=\bigoplus_{\nu\in Q^+}U_{v,t,\nu}^{+},
 \qquad
 U_{v,t}^{-}=\bigoplus_{\nu\in Q^+}U_{v,t,-\nu}^{-},
 \qquad
 Q^+=\N[I]=\sum_{i\in I}\N i.
\]
Thus $Q^+$ is the additive monoid of nonnegative integral combinations of
simple roots, and the two displayed subscripts record root degree.

% P2.11
For $\alpha=\sum_i a_i i$ and $\beta=\sum_i b_i i$ in $Q$, put
$K_\alpha=\prod_iK_i^{a_i}$ and
$K_\beta'=\prod_i(K_i')^{b_i}$, and define
\begin{equation}\label{eq:toral-pairing-character}
 \{\alpha,\beta\}
 =v^{\alpha\cdot\beta}
  t^{\langle\beta,\alpha\rangle-\langle\alpha,\beta\rangle}.
\end{equation}
The braces in this formula denote a scalar in $\K_D$ and do not denote a
set.
A \emph{skew-Hopf pairing} in the convention used here is a bilinear map
\[
 (\ ,\ )_\phi:U_{v,t}^{\geq0}\times U_{v,t}^{\leq0}\longrightarrow\K_D
\]
such that, for all elements for which the expressions are defined,
\begin{align}
 (1,y)_\phi&=\varepsilon(y),&
 (x,1)_\phi&=\varepsilon(x),
 \label{eq:skew-Hopf-units}\\
 (xx',y)_\phi&=(x\otimes x',\Delta^{\mathrm{op}}(y))_\phi,&
 (x,yy')_\phi&=(\Delta(x),y\otimes y')_\phi.
 \label{eq:skew-Hopf-products}
\end{align}
Here $\Delta^{\mathrm{op}}=P\circ\Delta$ is the opposite coproduct,
$P(a\otimes b)=b\otimes a$, and the pairing of simple tensors is defined by
$(x_1\otimes x_2,y_1\otimes y_2)_\phi
=(x_1,y_1)_\phi(x_2,y_2)_\phi$.
The adjective ``skew'' refers precisely to the opposite coproduct in the
first multiplication identity.
Fan and Xing prove that there is a unique such pairing with generator values
\begin{align}
 (K_\alpha,K_\beta')_\phi&=\{\alpha,\beta\},&
 (E_i,F_j)_\phi&=\frac{\delta_{ij}}{v_i^{-1}-v_i},
 \label{eq:skew-Hopf-generators}\\
 (K_\alpha,F_j)_\phi&=0,&
 (E_i,K_\beta')_\phi&=0.
 \label{eq:skew-Hopf-mixed-zero}
\end{align}
This is the skew-Hopf pairing used by Fan, Ma, and Xing
\cite[Proposition~4]{fan-xing2019}.

% P2.12
For every $\nu\in Q^+$, the restriction
\begin{equation}\label{eq:homogeneous-nondegeneracy}
 (\ ,\ )_\phi:
 U_{v,t,\nu}^{+}\times U_{v,t,-\nu}^{-}\longrightarrow\K_D
\end{equation}
is nondegenerate, meaning that an element in either factor is zero whenever
it pairs to zero with every element in the other factor
\cite{fan-ma-xing2024}.
Choose a basis $\mathcal B_\nu$ of $U_{v,t,-\nu}^{-}$.
Nondegeneracy gives a unique paired-dual basis
$\{b^\sharp:b\in\mathcal B_\nu\}$ of $U_{v,t,\nu}^{+}$ satisfying
\begin{equation}\label{eq:paired-dual-basis}
 (b^\sharp,c)_\phi=\delta_{b,c}
 \qquad(b,c\in\mathcal B_\nu),
\end{equation}
where $\delta_{b,c}$ is $1$ when $b=c$ and $0$ otherwise.
The symbol $\sharp$ is used here to avoid confusing this paired dual with
the ordinary vector-space dual denoted later by a star.
Define
\begin{equation}\label{eq:theta}
 \Theta_{t,\nu}=\sum_{b\in\mathcal B_\nu}b\otimes b^\sharp
 \in U_{v,t,-\nu}^{-}\otimes U_{v,t,\nu}^{+},
 \qquad
 \Theta_t=\sum_{\nu\in Q^+}\Theta_{t,\nu}.
\end{equation}
The tensor $\Theta_{t,\nu}$ is independent of the chosen basis, because it
is the canonical tensor of the nondegenerate pairing.

% P2.13
The sum $\Theta_t$ belongs to the root-degree completion
\begin{equation}\label{eq:root-degree-completion}
 \widehat{U_{v,t}^{-}\otimes U_{v,t}^{+}}
 =\prod_{\nu\in Q^+}
   U_{v,t,-\nu}^{-}\otimes U_{v,t,\nu}^{+}.
\end{equation}
The product sign means that one component of every root degree may be
nonzero, whereas the ordinary direct sum would permit only finitely many
nonzero components.
On a finite-dimensional weight module, $\Theta_{t,\nu}$ lowers the weight in
the first tensor factor by $\nu$ and raises the weight in the second factor
by $\nu$.
Only finitely many $\nu$ can act nontrivially on a fixed input tensor, so
the completed sum defines an ordinary linear operator on every tensor
product used below.
Fan, Ma, and Xing call $\Theta_t$ the quasi-$R$-matrix
\cite{fan-ma-xing2024}.
The prefix ``quasi'' indicates that the toral diagonal factor and the flip
of tensor factors have not yet been included.

% P2.14
The homogeneous components satisfy
\begin{align}
 (E_i\otimes1)\Theta_{t,\nu}+(K_i\otimes E_i)\Theta_{t,\nu-i}
 &=\Theta_{t,\nu}(E_i\otimes1)
 +\Theta_{t,\nu-i}(K_i'\otimes E_i),
 \label{eq:theta-E}\\
 (1\otimes F_i)\Theta_{t,\nu}+(F_i\otimes K_i')\Theta_{t,\nu-i}
 &=\Theta_{t,\nu}(1\otimes F_i)
 +\Theta_{t,\nu-i}(F_i\otimes K_i).
 \label{eq:theta-F}
\end{align}
Here $\Theta_{t,0}=1\otimes1$ and $\Theta_{t,\eta}=0$ when
$\eta\notin Q^+$.
These identities follow from the skew-Hopf axioms and the paired-dual basis
construction \cite{fan-ma-xing2024}.
They are the relations used in Section~\ref{sec:transport} to compare
$\Theta_t$ with its specialization at $t=1$.

\subsection{Review of the Fan--Ma--Xing link invariant}

% P2.15
An oriented tangle of type $(k,l)$ is a finite disjoint union of oriented
arcs and circles embedded in the slab $\mathbb R^2\times[0,1]$, with $k$
prescribed endpoints on the bottom plane and $l$ prescribed endpoints on
the top plane, considered up to ambient isotopy that fixes the boundary
\cite[Section~X.5]{kassel1995}.
Here $\mathbb R$ is the field of real numbers, an arc is an embedded copy of
an interval, and ambient isotopy means a continuous deformation of the whole
slab that leaves its boundary fixed.
The bottom and top orientations determine words in the signs $+$ and $-$:
the sign records whether the local orientation agrees with the upward
direction or points in the opposite direction.
The category $\mathsf{OTa}$ has these finite sign words as objects and
isotopy classes of oriented tangles with the indicated bottom and top words
as morphisms.
A category consists of objects and composable arrows called morphisms,
together with identity morphisms and an associative composition law.
Vertical stacking is composition, horizontal juxtaposition is tensor
product, and the empty word $\varnothing$ is the tensor unit.
A monoidal category is called \emph{strict} when repeated tensor products
may be written without associativity or unit isomorphisms; with the preceding
operations, $(\mathsf{OTa},\otimes,\varnothing)$ is strict
\cite[Proposition~XII.2.1]{kassel1995}.
Apart from identity strands, it is generated by six types of morphisms:
the positive crossing $X_+$, the negative crossing $X_-$, and the four
oriented cups and caps \cite[Theorem~3.2]{turaev1990}.
There is no independent full-twist generator in this ordinary oriented
tangle category; the Reidemeister-I curl is one of the relations that the
six assigned operators must satisfy.
A Reidemeister-I curl is the local move that adds or removes one small kink
in a strand without changing the remainder of the diagram.

% P2.16
Fix a finite-dimensional $L$-weight $U_{v,t}$-module $M$ and put
$M(+)=M$ and $M(-)=M^*$.
For a word $\boldsymbol\epsilon=(\epsilon_1,\ldots,\epsilon_r)$, define
\begin{equation}\label{eq:signed-word-module}
 M(\boldsymbol\epsilon)
 =M(\epsilon_1)\otimes\cdots\otimes M(\epsilon_r),
 \qquad M(\varnothing)=\K_D.
\end{equation}
The target category has the spaces $M(\boldsymbol\epsilon)$ as objects and
$U_{v,t}$-module homomorphisms as morphisms; the tensor product action uses
the coproduct in \eqref{eq:coproduct-torus}--\eqref{eq:coproduct-roots}.
A $U_{v,t}$-module homomorphism $f:P\to Q$ is a $\K_D$-linear map satisfying
$f(u p)=u f(p)$ for every $u\in U_{v,t}$ and $p\in P$.
To \emph{color} a tangle means to assign a representation to each connected
component and to use its dual when the local orientation is reversed.
A knot has exactly one connected component, whereas a link may have more
than one.
The signs $+$ and $-$ on local strands therefore indicate $M$ and $M^*$;
they do not indicate two different colors.
The Fan--Ma--Xing functor reviewed here fixes one module $M$, so every link
component receives the same color $M$ \cite{fan-ma-xing2024}.

% P2.17
For weights $\lambda,\mu\in L$, define the toral bicharacter
\begin{equation}\label{eq:toral-bicharacter}
 f_t(\lambda,\mu)
 =v^{-\lambda\cdot\mu}
  t^{\langle\lambda,\mu\rangle-\langle\mu,\lambda\rangle}.
\end{equation}
The word bicharacter means that $f_t$ is multiplicative in each argument
with respect to addition in $L$.
Let $P_{M,N}:M\otimes N\to N\otimes M$ be the flip
$P_{M,N}(m\otimes n)=n\otimes m$.
On $N\otimes M$, let $\widetilde f_t^{N,M}$ be the diagonal map
\begin{equation}\label{eq:toral-operator}
 \widetilde f_t^{N,M}(n_\mu\otimes m_\lambda)
 =f_t(\mu,\lambda)n_\mu\otimes m_\lambda.
\end{equation}
Let $\Theta_t^{N,M}$ denote the action of the completed sum
\eqref{eq:theta} on $N\otimes M$.
The operator without a flip and the checked operator are, respectively,
\begin{align}
 R_t^{N,M}
 &=\Theta_t^{N,M}\widetilde f_t^{N,M}:
 N\otimes M\longrightarrow N\otimes M,
 \label{eq:unchecked-R-definition}\\
 \widehat R_t^{M,N}
 &=R_t^{N,M}P_{M,N}
 =\Theta_t^{N,M}\widetilde f_t^{N,M}P_{M,N}:
 M\otimes N\longrightarrow N\otimes M.
 \label{eq:checked-R-definition}
\end{align}
Thus $R_t^{N,M}$ does not exchange the two tensor factors, while
$\widehat R_t^{M,N}$ does exchange them.
The hat is used in this paper to record the included flip; the map denoted
$\mathcal R$ by Fan, Ma, and Xing is our $\widehat R$
\cite{fan-ma-xing2024}.
Both maps are invertible, and their inverses have the explicitly ordered
forms
\begin{align}
 (R_t^{N,M})^{-1}
 &=(\widetilde f_t^{N,M})^{-1}(\Theta_t^{N,M})^{-1}:
 N\otimes M\longrightarrow N\otimes M,
 \label{eq:unchecked-R-inverse}\\
 (\widehat R_t^{M,N})^{-1}
 &=P_{N,M}(\widetilde f_t^{N,M})^{-1}(\Theta_t^{N,M})^{-1}:
 N\otimes M\longrightarrow M\otimes N.
 \label{eq:checked-R-inverse}
\end{align}
The order in these formulas is forced by
$(ABC)^{-1}=C^{-1}B^{-1}A^{-1}$.
Fan, Ma, and Xing prove that $\widehat R_t^{M,N}$ is a
$U_{v,t}$-module isomorphism and that these operators satisfy the braid
relation \cite{fan-ma-xing2024}.

% P2.18
Let $\mathcal E$ be a homogeneous basis of $M$ and let
$\mathcal E^*=\{b^*:b\in\mathcal E\}$ be its ordinary dual basis.
The star in $M^*$ and in $b^*$ denotes ordinary vector-space duality and is
unrelated to the paired-dual symbol $b^\sharp$ in \eqref{eq:theta}.
If $b\in\mathcal E$ has weight $\lambda$, then $b^*$ has weight
$-\lambda$ under the dual action defined in Subsection~2.3.
The four maps assigned to the four oriented caps and cups are
\begin{align}
 \ev &:M^*\otimes M\to\K_D,
 &\ev(b^*\otimes m)&=b^*(m),
 \label{eq:ev}\\
 \qtr &:M\otimes M^*\to\K_D,
 &\qtr(m_\lambda\otimes \varphi)&=v_{-\lambda}^{2}\varphi(m_\lambda),
 \label{eq:qtr}\\
 \coev &:\K_D\to M^*\otimes M,
 &\coev(1)&=\sum_{b\in\mathcal E}v_{\wt(b)}^{2}b^*\otimes b,
 \label{eq:coev}\\
 \coqtr &:\K_D\to M\otimes M^*,
 &\coqtr(1)&=\sum_{b\in\mathcal E}b\otimes b^*.
 \label{eq:coqtr}
\end{align}
The abbreviations $\ev$ and $\coev$ mean evaluation and coevaluation.
The maps $\qtr$ and $\coqtr$ are the evaluation and coevaluation in the
opposite orientation; the letter ``q'' records the weight-dependent quantum
factor in $\qtr$.
All four maps are $U_{v,t}$-module homomorphisms and are independent of the
chosen homogeneous basis in the basis-dependent formulas
\cite{fan-ma-xing2024}.

% P2.19
Assume now that $M$ is a simple highest-weight module of highest weight
$\Lambda$.
Define the curl-normalization scalar and normalized crossing operators by
\begin{equation}\label{eq:curl}
 a_M(t)=f_t(\Lambda,\Lambda)v_{-\Lambda}^{2},
 \qquad
 C_t^+=a_M(t)^{-1}\widehat R_t^{M,M},
 \qquad
 C_t^-=a_M(t)(\widehat R_t^{M,M})^{-1},
\end{equation}
The operators $C_t^+$ and $C_t^-$ are assigned to $X_+$ and $X_-$,
respectively.
This normalization makes the operator evaluation of a Reidemeister-I curl
equal to the identity; it does not add a new twist operator
\cite{fan-ma-xing2024}.
Together with identity maps and the four maps in
\eqref{eq:ev}--\eqref{eq:coqtr}, these crossing operators satisfy every
defining relation of $\mathsf{OTa}$ \cite{fan-ma-xing2024}.
Consequently, they define a strict monoidal functor
\begin{equation}\label{eq:fmx-tangle-functor}
 \mathcal T_t:(\mathsf{OTa},\otimes,\varnothing)
 \longrightarrow(\mathsf{Mod}_M,\otimes,\K_D),
\end{equation}
where $\mathsf{Mod}_M$ denotes the module category described after
\eqref{eq:signed-word-module}.
A strict monoidal functor assigns an object to every object and a morphism to
every morphism, preserves identities and composition, and preserves tensor
products and the tensor unit exactly.
As shown in Figure~\ref{fig:tangle-generators}, the functor's complete
assignment on the two identity strands and the six nonidentity generators is
recorded; every diagram is read from its source word at the bottom to its
target word at the top.

\begin{figure}[t]
\centering
\includegraphics[width=\textwidth]{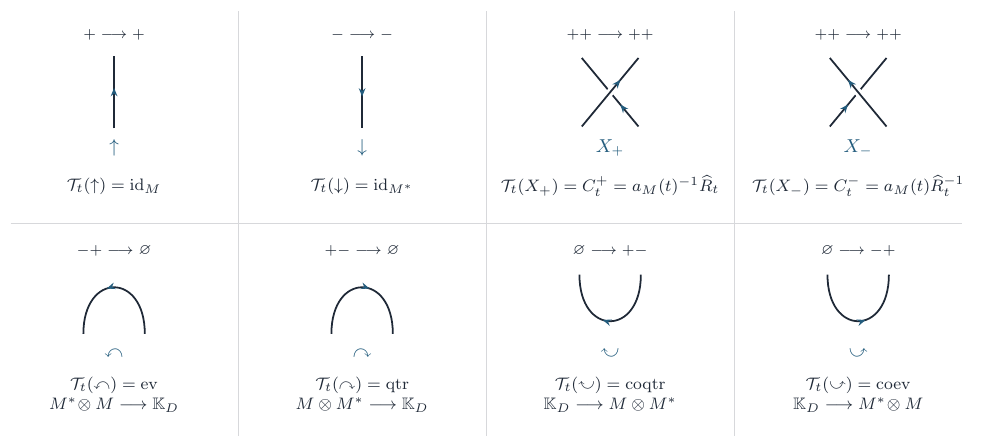}
\caption{The elementary oriented-tangle diagrams and their images under
$\mathcal T_t$; the small arrows on the strands record their orientations.}
\label{fig:tangle-generators}
\end{figure}

In symbols, the object assignment and the generator assignment are
\begin{equation}\label{eq:tangle-generator-dictionary}
\begin{aligned}
 \mathcal T_t(+)&=M,&
 \mathcal T_t(-)&=M^*,&
 \mathcal T_t(\varnothing)&=\K_D,\\
 \mathcal T_t(\uparrow)&=\id_M,&
 \mathcal T_t(\downarrow)&=\id_{M^*},\\
 \mathcal T_t(\curvearrowleft)&=\ev,&
 \mathcal T_t(\curvearrowright)&=\qtr,\\
 \mathcal T_t(\undercurvearrowleft)&=\coqtr,&
 \mathcal T_t(\undercurvearrowright)&=\coev,\\
 \mathcal T_t(X_+)&=a_M(t)^{-1}\widehat R_t^{M,M},&
 \mathcal T_t(X_-)&=a_M(t)(\widehat R_t^{M,M})^{-1}.
\end{aligned}
\end{equation}
The source and target spaces of the four cup-cap maps are displayed in
\eqref{eq:ev}--\eqref{eq:coqtr}, and both crossing operators in
\eqref{eq:tangle-generator-dictionary} map $M\otimes M$ to $M\otimes M$.
Crossings involving a downward-oriented strand are obtained by composing the
displayed crossing with the cup-cap maps, so they are derived morphisms rather
than additional generators.
For composable tangles $T:\mathbf X\to\mathbf Y$ and
$S:\mathbf Y\to\mathbf Z$, and for arbitrary tangles $T$ and $U$, strict
monoidality means
\begin{equation}\label{eq:tangle-functor-operations}
 \mathcal T_t(S\circ T)=\mathcal T_t(S)\circ\mathcal T_t(T),
 \qquad
 \mathcal T_t(T\otimes U)=\mathcal T_t(T)\otimes\mathcal T_t(U).
\end{equation}
Thus vertical layers are evaluated from bottom to top, while the corresponding
operator composition in \eqref{eq:tangle-functor-operations} is read from
right to left.
For a closed oriented link $\mathcal L:\varnothing\to\varnothing$,
$\mathcal T_t(\mathcal L)$ is an endomorphism of the one-dimensional tensor
unit $\K_D$ and hence multiplication by a unique scalar.
We denote that scalar by
\begin{equation}\label{eq:invariant-definition}
 \mathcal T_t(\mathcal L)
 =I_{v,t}^{M}(\mathcal L)\id_{\K_D},
 \qquad I_{v,t}^{M}(\mathcal L)\in\K_D.
\end{equation}
This scalar is the normalized two-parameter oriented-link invariant compared
with its $t=1$ specialization in the next section.

%==============================================================================
\section{Parameter Transport and Equality of Link Invariants}
\label{sec:transport}
%==============================================================================

\subsection{The meaning of transport and the transported module}

% P3.01
The word \emph{transport} is used in this paper for the explicit transfer of
module and tangle-operator data given below; it is not the name of an
additional structure already contained in a quantum algebra.
The complete transport data consist of a change of the generator actions, a
comparison map for tensor products, a comparison map for algebraic duals, and
intertwining identities for the six elementary oriented-tangle operators.
An algebra isomorphism is a bijective linear map between two algebras that
preserves multiplication and the identity element.
In particular, no algebra isomorphism
\begin{equation}\label{eq:no-algebra-isomorphism}
 U_{v,1}\cong U_{v,t}
\end{equation}
is asserted or used.
The conclusion will instead compare the relevant module categories and their
tangle functors by explicit invertible linear maps.

% P3.02
Let \(M=\bigoplus_{\lambda\in L}M_\lambda\) be an \(L\)-weight
\(U_{v,1}\)-module, where \(U_{v,1}\) has already been extended to the common
field \(\K_D\) as explained after \eqref{eq:antipode-torus}.
On the same vector space and the same weight decomposition define
\begin{align}
 K_i^{(t)}m_\lambda
 &=v^{i\cdot\lambda}
   t^{\langle\lambda,i\rangle-\langle i,\lambda\rangle}m_\lambda,
 &
 (K_i')^{(t)}m_\lambda
 &=v^{-i\cdot\lambda}
   t^{\langle\lambda,i\rangle-\langle i,\lambda\rangle}m_\lambda,
 \label{eq:transport-torus}\\
 E_i^{(t)}m_\lambda
 &=t^{\langle\lambda,i\rangle}E_i^{(1)}m_\lambda,
 &
 F_i^{(t)}m_\lambda
 &=t^{d_i-\langle i,\lambda\rangle}F_i^{(1)}m_\lambda.
 \label{eq:transport-roots}
\end{align}
The superscripts \((1)\) and \((t)\) distinguish the one-parameter and
two-parameter actions and do not denote divided powers.
The vector space equipped with these new actions is denoted by \(\Phi_t(M)\).

\begin{theorem}[Module transport]\label{thm:module-transport}
Equations \eqref{eq:transport-torus}--\eqref{eq:transport-roots} define an
\(L\)-weight \(U_{v,t}\)-module on \(\Phi_t(M)\).
On a weight-preserving \(U_{v,1}\)-module homomorphism \(h:M\to N\), define
\(\Phi_t(h)=h\) on the underlying vector spaces.
This makes \(\Phi_t\) an equivalence between the corresponding
\(L\)-weight module categories; its inverse is obtained by replacing every
displayed power of \(t\) by its reciprocal.
Finite dimensionality, integrability, simplicity, and the highest-weight
property are preserved by this equivalence.
\end{theorem}

\begin{proof}
We verify the defining relations rather than infer them from an algebra
isomorphism.
The toral generators are diagonal and invertible by
\eqref{eq:transport-torus}.
For concise notation, let
\[
 k_i(\lambda)=v^{i\cdot\lambda}
 t^{\langle\lambda,i\rangle-\langle i,\lambda\rangle},
 \qquad
 k_i'(\lambda)=v^{-i\cdot\lambda}
 t^{\langle\lambda,i\rangle-\langle i,\lambda\rangle};
\]
these are the respective eigenvalues of \(K_i^{(t)}\) and
\((K_i')^{(t)}\) on the weight space \(M_\lambda\).
Because \(E_j^{(t)}\) raises a weight by \(j\) and \(F_j^{(t)}\) lowers a
weight by \(j\), the four required eigenvalue quotients are
\begin{equation}\label{eq:transport-toral-check}
\begin{aligned}
 \frac{k_i(\lambda+j)}{k_i(\lambda)}
 &=v^{i\cdot j}t^{\langle j,i\rangle-\langle i,j\rangle},&
 \frac{k_i'(\lambda+j)}{k_i'(\lambda)}
 &=v^{-i\cdot j}t^{\langle j,i\rangle-\langle i,j\rangle},\\
 \frac{k_i(\lambda-j)}{k_i(\lambda)}
 &=v^{-i\cdot j}t^{\langle i,j\rangle-\langle j,i\rangle},&
 \frac{k_i'(\lambda-j)}{k_i'(\lambda)}
 &=v^{i\cdot j}t^{\langle i,j\rangle-\langle j,i\rangle}.
\end{aligned}
\end{equation}
The two quotients in the first row prove the two relations in
\eqref{eq:torus-E}, and the two quotients in the second row prove the two
relations in \eqref{eq:torus-F}.

For \(i\ne j\), each of
\(E_i^{(t)}F_j^{(t)}m_\lambda\) and
\(F_j^{(t)}E_i^{(t)}m_\lambda\) is its one-parameter counterpart multiplied by
\begin{equation}\label{eq:off-diagonal-factor}
 t^{d_j-\langle j,\lambda\rangle+
       \langle\lambda,i\rangle-\langle j,i\rangle}.
\end{equation}
For \(i=j\), their common multiplier is
\begin{equation}\label{eq:diagonal-factor}
 t^{\langle\lambda,i\rangle-\langle i,\lambda\rangle}.
\end{equation}
Multiplying the one-parameter commutator by
\eqref{eq:diagonal-factor} gives \eqref{eq:commutator}, because the same
factor multiplies \(K_i^{(1)}-(K_i')^{(1)}\).

It remains to check the two Serre relations.
Repeated use of \eqref{eq:transport-roots} and the factorial identity
\eqref{eq:factorial} gives
\begin{align}
 (E_i^{(t)})^{(p)}m_\lambda
 &=t^{p\langle\lambda,i\rangle}
   (E_i^{(1)})^{(p)}m_\lambda,
 \label{eq:transport-E-divided}\\
 (F_i^{(t)})^{(p)}m_\lambda
 &=t^{-p\langle i,\lambda\rangle+pd_i}
   (F_i^{(1)})^{(p)}m_\lambda.
 \label{eq:transport-F-divided}
\end{align}
Indeed, before division by the factorial, the \(E_i\)-exponent is
\(p\langle\lambda,i\rangle+d_ip(p-1)/2\), and the factorial removes the
second summand.
For \(F_i\), the corresponding exponent before division is
\(-p\langle i,\lambda\rangle+d_ip(p+1)/2\), and the factorial again removes
\(d_ip(p-1)/2\).

Fix \(i\ne j\), write \(N=N_{ij}\), \(A=A_{ij}\), and \(B=B_{ij}\), and let
\(p+p'=N\).
After including the coefficient \(\kappa_{ij}(p,p')\), the term
\[
 (E_i^{(t)})^{(p)}E_j^{(t)}
 (E_i^{(t)})^{(p')}m_\lambda
\]
has total \(t\)-exponent
\begin{equation}\label{eq:positive-Serre-common}
 N\langle\lambda,i\rangle+\langle\lambda,j\rangle+NA,
\end{equation}
which is independent of \(p\).
Similarly, the term
\[
 (F_i^{(t)})^{(p')}F_j^{(t)}
 (F_i^{(t)})^{(p)}m_\lambda
\]
has total \(t\)-exponent
\begin{equation}\label{eq:negative-Serre-common}
 -N\langle i,\lambda\rangle-\langle j,\lambda\rangle
 +Nd_i+d_j+NA,
\end{equation}
which is also independent of \(p\).
Thus each two-parameter Serre sum is a nonzero common scalar times its
one-parameter Serre sum and is zero.
All defining relations have now been checked.

A weight-preserving homomorphism commutes with the four transported actions
because every added factor depends only on the input weight.
The reciprocal formulas give a two-sided inverse functor.
Equations \eqref{eq:transport-E-divided}--\eqref{eq:transport-F-divided} show
that local nilpotence is preserved, and
\eqref{eq:transport-roots} shows that highest-weight vectors are unchanged.
Simplicity is preserved by an equivalence, and finite dimensionality is
unchanged because the underlying vector space is unchanged.
\end{proof}

\subsection{The tensor comparison}

% P3.03
For \(U_{v,1}\) \(L\)-weight modules \(M,N\), define
\begin{equation}\label{eq:tensorator}
 \begin{split}
 J_{M,N}:\;&
 \Phi_t(M)\otimes_{\Delta_t}\Phi_t(N)
 \longrightarrow
 \Phi_t(M\otimes_{\Delta_1}N),\\
 &J_{M,N}(m_\lambda\otimes n_\mu)
 =t^{\langle\mu,\lambda\rangle}
  m_\lambda\otimes n_\mu.
 \end{split}
\end{equation}
The subscripts \(\Delta_t\) and \(\Delta_1\) specify which coproduct supplies
the tensor-product action.
The map \(J_{M,N}\) is called the \emph{tensor comparison} or
\emph{tensorator}: it compares the tensor product formed after transport with
the transport of the tensor product formed before transport.

\begin{proposition}[Strong monoidal comparison]\label{prop:tensorator}
The map \(J_{M,N}\) is an invertible \(U_{v,t}\)-module homomorphism.
For a third module \(P\), the maps satisfy
\begin{equation}\label{eq:tensorator-coherence}
 J_{M\otimes N,P}(J_{M,N}\otimes\id_P)
 =
 J_{M,N\otimes P}(\id_M\otimes J_{N,P}),
\end{equation}
and \(J\) is the identity if either factor is the one-dimensional
weight-zero trivial module \(\K_D\).
Consequently, \((\Phi_t,J)\) is a strong monoidal equivalence.
The term \emph{strong monoidal} means exactly that the object comparison is
invertible, satisfies \eqref{eq:tensorator-coherence}, and respects the tensor
unit.
\end{proposition}

\begin{proof}
On \(m_\lambda\otimes n_\mu\), the two summands in the \(E_i\)-intertwining
identity on the two sides acquire, respectively, the common factors
\begin{equation}\label{eq:tensorator-E-factors}
 t^{\langle\mu,\lambda\rangle+\langle\lambda+\mu,i\rangle},
 \qquad
 v^{i\cdot\lambda}
 t^{\langle\mu,\lambda\rangle+\langle\lambda+\mu,i\rangle}.
\end{equation}
The two \(F_i\)-summands acquire
\begin{equation}\label{eq:tensorator-F-factors}
 t^{\langle\mu,\lambda\rangle+d_i-\langle i,\lambda+\mu\rangle},
 \qquad
 v^{-i\cdot\mu}
 t^{\langle\mu,\lambda\rangle+d_i-\langle i,\lambda+\mu\rangle}.
\end{equation}
These equalities follow by substituting \eqref{eq:transport-roots} into
\(\Delta_t(E_i)\) and \(\Delta_t(F_i)\), so \(J_{M,N}\) intertwines
\(E_i\) and \(F_i\).
For \(K_i\), both sides multiply the tensor by
\begin{equation}\label{eq:tensorator-K-factor}
 v^{i\cdot(\lambda+\mu)}
 t^{\langle\lambda+\mu,i\rangle-\langle i,\lambda+\mu\rangle};
\end{equation}
for \(K_i'\), the exponent of \(v\) is negated.
Thus \(J_{M,N}\) is a module homomorphism.

On \(m_\lambda\otimes n_\mu\otimes p_\nu\), both sides of
\eqref{eq:tensorator-coherence} multiply by
\begin{equation}\label{eq:triple-tensorator-factor}
 t^{\langle\mu,\lambda\rangle+
    \langle\nu,\lambda\rangle+\langle\nu,\mu\rangle}.
\end{equation}
This proves coherence.
Negating the exponent in \eqref{eq:tensorator} gives the inverse, and
bilinearity gives
\(t^{\langle0,\lambda\rangle}=t^{\langle\lambda,0\rangle}=1\) for the tensor
unit.
\end{proof}

\subsection{Transport of the quasi-\texorpdfstring{$R$}{R}-matrix and crossings}

% P3.04
The next lemma compares the canonical tensors in \eqref{eq:theta}.
The notation \(\Theta_{1,\nu}\) means the component obtained at \(t=1\),
extended to \(\K_D\), and acting through the one-parameter module structures.
For a word \(\mathbf i=(i_1,\ldots,i_r)\) in \(I\), write
\(|\mathbf i|=i_1+\cdots+i_r\), and define two word exponents by
\begin{equation}\label{eq:word-exponents}
 c_+(\mathbf i)=\sum_{1\leq a<b\leq r}\langle i_b,i_a\rangle,
 \qquad
 c_-(\mathbf i)=\sum_{a=1}^r d_{i_a}
 +\sum_{1\leq a<b\leq r}\langle i_a,i_b\rangle.
\end{equation}
For \(\nu\in Q^+\), define graded linear comparison maps
\begin{align}
 \chi_\nu^+ &:U_{v,t,\nu}^+\longrightarrow U_{v,1,\nu}^+,
 &\chi_\nu^+(E_{i_1}\cdots E_{i_r})
 &=t^{c_+(\mathbf i)}E_{i_1}\cdots E_{i_r},
 \label{eq:chi-plus}\\
 \chi_\nu^- &:U_{v,t,-\nu}^-\longrightarrow U_{v,1,-\nu}^-,
 &\chi_\nu^-(F_{i_1}\cdots F_{i_r})
 &=t^{c_-(\mathbf i)}F_{i_1}\cdots F_{i_r},
 \label{eq:chi-minus}
\end{align}
where \(|\mathbf i|=\nu\), and the generators on the right are the
one-parameter generators.
These maps are linear comparisons of graded vector spaces; they are not
asserted to preserve multiplication.

\begin{lemma}[Graded pairing comparison]\label{lem:graded-pairing-comparison}
The formulas \eqref{eq:chi-plus}--\eqref{eq:chi-minus} descend from words to
well-defined linear isomorphisms on the displayed homogeneous subspaces.
For \(x\in U_{v,t,\nu}^+\), \(y\in U_{v,t,-\nu}^-\),
\(m_\lambda\in M_\lambda\), and \(n_\mu\in N_\mu\), they satisfy
\begin{align}
 x^{(t)}m_\lambda
 &=t^{\langle\lambda,\nu\rangle}
   \chi_\nu^+(x)^{(1)}m_\lambda,
 \label{eq:chi-plus-action}\\
 y^{(t)}n_\mu
 &=t^{-\langle\nu,\mu\rangle}
   \chi_\nu^-(y)^{(1)}n_\mu,
 \label{eq:chi-minus-action}\\
 (x,y)_{\phi,t}
 &=t^{-\langle\nu,\nu\rangle}
   \bigl(\chi_\nu^+(x),\chi_\nu^-(y)\bigr)_{\phi,1}.
 \label{eq:pairing-transport}
\end{align}
The subscripts on the pairing distinguish the two-parameter pairing from its
\(t=1\) specialization.
\end{lemma}

\begin{proof}
In the word \(E_{i_1}\cdots E_{i_r}\), the rightmost generator acts first.
Repeated use of \eqref{eq:transport-roots} therefore gives the exponent
\begin{equation}\label{eq:positive-word-action}
 \sum_{a=1}^r
 \left\langle\lambda+\sum_{b>a}i_b,i_a\right\rangle
 =\langle\lambda,\nu\rangle+c_+(\mathbf i),
\end{equation}
which proves \eqref{eq:chi-plus-action} on words.
The same calculation for \(F_{i_1}\cdots F_{i_r}\) gives
\begin{equation}\label{eq:negative-word-action}
 \sum_{a=1}^r
 \left(d_{i_a}-
 \left\langle i_a,\mu-\sum_{b>a}i_b\right\rangle\right)
 =-\langle\nu,\mu\rangle+c_-(\mathbf i),
\end{equation}
which proves \eqref{eq:chi-minus-action} on words.

It remains to check that the word formulas respect the defining relations of
the positive and negative halves.
The factorial identity \eqref{eq:factorial} first gives

\begin{equation}\label{eq:chi-divided-powers}
 \chi_{pi}^+\bigl(E_i^{(p)}\bigr)=(E_i^{(p)})_{t=1},
 \qquad
 \chi_{pi}^-\bigl(F_i^{(p)}\bigr)=t^{pd_i}(F_i^{(p)})_{t=1}.
\end{equation}
For \(p+p'=N_{ij}\), direct substitution of
\eqref{eq:word-exponents}, \eqref{eq:kappa}, and
\eqref{eq:chi-divided-powers} shows that every summand of the positive Serre
relation is mapped to \(t^{N_{ij}A_{ij}}\) times the corresponding summand of
the one-parameter Serre relation.
Every summand of the negative Serre relation is mapped to
\(t^{N_{ij}d_i+d_j+N_{ij}A_{ij}}\) times its one-parameter counterpart.
The factors are independent of \(p\), so both Serre relations are preserved.
Applying the same construction with all exponents negated gives inverse maps;
hence \(\chi_\nu^+\) and \(\chi_\nu^-\) are well-defined isomorphisms.

We now prove \eqref{eq:pairing-transport} without suppressing the coproduct
calculation.
For words \(\mathbf i=(i_1,\ldots,i_r)\) and
\(\mathbf j=(j_1,\ldots,j_r)\) of the same root degree \(\nu\), put
\[
 P_t(\mathbf i,\mathbf j)
 =\bigl(E_{i_1}\cdots E_{i_r},
        F_{j_1}\cdots F_{j_r}\bigr)_{\phi,t},
 \qquad g_i=(v_i^{-1}-v_i)^{-1}.
\]
Write \(\{\alpha,\beta\}_t\) for the scalar in
\eqref{eq:toral-pairing-character} and
\(\{\alpha,\beta\}_1\) for its value at \(t=1\).
Let \(i=i_r\), let \(\mathbf i'=(i_1,\ldots,i_{r-1})\), and let
\(\widehat{\mathbf j}_k\) denote the word obtained from \(\mathbf j\) by
deleting \(j_k\).
The skew-Hopf rule gives the exact recursion
\begin{equation}\label{eq:word-pairing-recurrence}
 P_t(\mathbf i,\mathbf j)
 =g_i\sum_{\substack{1\leq k\leq r\\ j_k=i}}
   \left(\prod_{\ell>k}\{j_\ell,i\}_t\right)
   P_t(\mathbf i',\widehat{\mathbf j}_k).
\end{equation}
Indeed, before applying the flip \(P\) in
\(\Delta^{\mathrm{op}}=P\circ\Delta\), the relevant term of
\(\Delta(F_{j_1}\cdots F_{j_r})\) must choose
\(F_i\otimes K_i'\) in position \(k\) and
\(1\otimes F_{j_\ell}\) in every other position.
After the flip, its second tensor factor is \(F_i\), while its first tensor
factor is the word obtained from \(F_{j_1}\cdots F_{j_r}\) by replacing the
\(k\)th letter with \(K_i'\); hence it is the only kind of term that can pair
nontrivially with
\(E_{i_1}\cdots E_{i_{r-1}}\otimes E_i\).
Moving \(K_i'\) to the right uses
\(K_i'F_{j_\ell}=\{j_\ell,i\}_tF_{j_\ell}K_i'\), producing the displayed
product; the final toral factor pairs with the counit component and contributes
one.
The last tensor factor contributes
\((E_i,F_i)_{\phi,t}=g_i\).

We prove the word identity
\begin{equation}\label{eq:word-pairing-transport}
 P_t(\mathbf i,\mathbf j)
 =t^{A(\mathbf i,\mathbf j)}P_1(\mathbf i,\mathbf j),
 \qquad
 A(\mathbf i,\mathbf j)
 =-\langle\nu,\nu\rangle+c_+(\mathbf i)+c_-(\mathbf j),
\end{equation}
by induction on the common word length \(r\).
For \(r=0\), both pairings of the empty words equal one and the exponent is
zero.
Fix a summand of \eqref{eq:word-pairing-recurrence}, set
\(\mu=\nu-i\), and apply the induction hypothesis to the two shortened
words.
The ratio of the toral products at \(t\) and at \(t=1\) is
\begin{equation}\label{eq:toral-product-ratio}
 \prod_{\ell>k}
 \frac{\{j_\ell,i\}_t}{\{j_\ell,i\}_1}
 =t^{T_k},
 \qquad
 T_k=\sum_{\ell>k}
 \bigl(\langle i,j_\ell\rangle-\langle j_\ell,i\rangle\bigr).
\end{equation}
On the other hand, bilinearity and \(\langle i,i\rangle=d_i\) give
\begin{align}
 &A(\mathbf i,\mathbf j)
   -A(\mathbf i',\widehat{\mathbf j}_k)\notag\\
 &=-\bigl(\langle\mu,i\rangle+\langle i,\mu\rangle+d_i\bigr)
   +\langle i,\mu\rangle+d_i
   +\sum_{\ell<k}\langle j_\ell,i\rangle
   +\sum_{\ell>k}\langle i,j_\ell\rangle\notag\\
 &=\sum_{\ell>k}
   \bigl(\langle i,j_\ell\rangle-\langle j_\ell,i\rangle\bigr)
 =T_k.
 \label{eq:pairing-exponent-difference}
\end{align}
Thus every summand in the \(t\)-recursion is
\(t^{A(\mathbf i,\mathbf j)}\) times the corresponding summand in the
\(t=1\) recursion, proving \eqref{eq:word-pairing-transport}.
Finally, the two comparison maps contribute
\(t^{c_+(\mathbf i)+c_-(\mathbf j)}\) to the one-parameter pairing.
Equation \eqref{eq:word-pairing-transport} is therefore precisely
\eqref{eq:pairing-transport} for monomials, and bilinearity extends it to all
homogeneous \(x\) and \(y\).
\end{proof}

\begin{lemma}[Quasi-\(R\) transport]\label{lem:theta-transport}
Let \(n_\mu\in N_\mu\), \(m_\lambda\in M_\lambda\), and
\(\nu\in Q^+\).
On the transported modules,
\begin{equation}\label{eq:theta-transport}
 \Theta_{t,\nu}(n_\mu\otimes m_\lambda)
 =
 t^{S_\nu(\mu,\lambda)}
 \Theta_{1,\nu}(n_\mu\otimes m_\lambda),
\end{equation}
where
\begin{equation}\label{eq:S}
 S_\nu(\mu,\lambda)
 =
 \langle\lambda,\nu\rangle-\langle\nu,\mu\rangle
 +\langle\nu,\nu\rangle.
\end{equation}
\end{lemma}

\begin{proof}
Choose a basis \(\mathcal B_{t,\nu}\) of
\(U_{v,t,-\nu}^-\), and use
\(\chi_\nu^-(\mathcal B_{t,\nu})\) as a basis of
\(U_{v,1,-\nu}^-\).
If \(b\in\mathcal B_{t,\nu}\) and \(b^\sharp\) is paired-dual to \(b\), then
\eqref{eq:pairing-transport} gives
\begin{equation}\label{eq:paired-basis-transport}
 \chi_\nu^+(b^\sharp)
 =t^{\langle\nu,\nu\rangle}
   \bigl(\chi_\nu^-(b)\bigr)^\sharp,
\end{equation}
where the sharp on the right denotes the paired-dual basis for the
one-parameter pairing.
Consequently,
\begin{equation}\label{eq:canonical-tensor-transport}
 (\chi_\nu^-\otimes\chi_\nu^+)(\Theta_{t,\nu})
 =t^{\langle\nu,\nu\rangle}\Theta_{1,\nu}.
\end{equation}
When \(\Theta_{t,\nu}\) acts on \(n_\mu\otimes m_\lambda\),
\eqref{eq:chi-minus-action} contributes
\(t^{-\langle\nu,\mu\rangle}\),
\eqref{eq:chi-plus-action} contributes
\(t^{\langle\lambda,\nu\rangle}\), and
\eqref{eq:canonical-tensor-transport} contributes
\(t^{\langle\nu,\nu\rangle}\).
Their product is \(t^{S_\nu(\mu,\lambda)}\), which proves
\eqref{eq:theta-transport}.
\end{proof}

\begin{theorem}[Checked \(R\)-matrix transport]
\label{thm:checked-R-transport}
For finite-dimensional \(L\)-weight \(U_{v,1}\)-modules \(M,N\), the checked
\(R\)-matrix and its inverse satisfy
\begin{align}
 \widehat R_t^{M,N}
 &=J_{N,M}^{-1}\widehat R_1^{M,N}J_{M,N},
 \label{eq:checked-R-transport}\\
 (\widehat R_t^{M,N})^{-1}
 &=J_{M,N}^{-1}(\widehat R_1^{M,N})^{-1}J_{N,M}.
 \label{eq:inverse-R-transport}
\end{align}
The source and target of \eqref{eq:checked-R-transport} are
\(\Phi_t(M)\otimes\Phi_t(N)\) and
\(\Phi_t(N)\otimes\Phi_t(M)\), respectively.
The source and target of \eqref{eq:inverse-R-transport} occur in the reverse
order.
\end{theorem}

\begin{proof}
Apply the right-hand side of \eqref{eq:checked-R-transport} to
\(m_\lambda\otimes n_\mu\) and retain the quasi-\(R\) term of root degree
\(\nu\).
The input tensor comparison contributes
\(t^{\langle\mu,\lambda\rangle}\).
After the flip, \(\Theta_{1,\nu}\) changes the two output weights to
\(\mu-\nu\) and \(\lambda+\nu\).
The inverse output comparison therefore contributes
\(t^{-\langle\lambda+\nu,\mu-\nu\rangle}\).
Their product factors as
\begin{align}
 t^{\langle\mu,\lambda\rangle-
    \langle\lambda+\nu,\mu-\nu\rangle}
 &=
 t^{\langle\mu,\lambda\rangle-\langle\lambda,\mu\rangle}
 t^{S_\nu(\mu,\lambda)}.
 \label{eq:R-factorization}
\end{align}
The first factor changes \(f_1(\mu,\lambda)\) into
\(f_t(\mu,\lambda)\) by \eqref{eq:toral-bicharacter}, and the second changes
\(\Theta_{1,\nu}\) into \(\Theta_{t,\nu}\) by
Lemma~\ref{lem:theta-transport}.
Summing over \(\nu\) proves \eqref{eq:checked-R-transport}.
Taking inverses and using the displayed source and target spaces proves
\eqref{eq:inverse-R-transport}.
\end{proof}

\subsection{Transport of algebraic duals and the four cup-cap maps}

% P3.05
Finite dimensionality is used here because the coevaluation maps in
\eqref{eq:coev}--\eqref{eq:coqtr} are finite sums over a basis.
Write $(-)^*$ for the operation of taking the algebraic dual with the action
in \eqref{eq:dual-action}.
Starting from a one-parameter module $M$, the two possible orders of dualizing
and transporting are
\begin{equation}\label{eq:two-dual-routes}
\begin{array}{ccccc}
 M&\xrightarrow{\ (-)^*\ }&M^*&\xrightarrow{\ \Phi_t\ }&\Phi_t(M^*),\\[3pt]
 M&\xrightarrow{\ \Phi_t\ }&\Phi_t(M)&\xrightarrow{\ (-)^*\ }&\Phi_t(M)^*.
\end{array}
\end{equation}
Thus the two endpoints \(\Phi_t(M^*)\) and \(\Phi_t(M)^*\) are a priori
different \(U_{v,t}\)-modules.
For a homogeneous functional \(\varphi_{-\lambda}\in(M^*)_{-\lambda}\),
define
\begin{equation}\label{eq:dual-identification}
 D_M:\Phi_t(M^*)\longrightarrow\Phi_t(M)^*,
 \qquad
 D_M(\varphi_{-\lambda})
 =t^{-\langle\lambda,\lambda\rangle}\varphi_{-\lambda}.
\end{equation}
Thus \(D_M\) is a weight-diagonal linear isomorphism; its subscript records
the original module whose two dual constructions are being compared.
Its purpose is precisely to compare the two endpoints of
\eqref{eq:two-dual-routes}.
All compositions of maps below are read from right to left.

\begin{proposition}[Dual comparison]\label{prop:dual-comparison}
The map \(D_M\) is a \(U_{v,t}\)-module isomorphism.
Under this identification, the four maps
\eqref{eq:ev}--\eqref{eq:coqtr} obey
\begin{align}
 \makebox[.43\linewidth][l]{$\displaystyle\ev_t(D_M\otimes\id_M)$}
 &=\ev_1J_{M^*,M},
 \label{eq:dual-evaluation-transport}\\
 \makebox[.43\linewidth][l]{$\displaystyle\qtr_t(\id_M\otimes D_M)$}
 &=\qtr_1J_{M,M^*},
 \label{eq:qtr-transport}\\
 \makebox[.43\linewidth][l]{$\displaystyle\coev_t$}
 &=(D_M\otimes\id_M)J_{M^*,M}^{-1}\coev_1,
 \label{eq:coev-transport}\\
 \makebox[.43\linewidth][l]{$\displaystyle\coqtr_t$}
 &=(\id_M\otimes D_M)J_{M,M^*}^{-1}\coqtr_1.
 \label{eq:coqtr-transport}
\end{align}
Every subscript \(1\) map in these formulas is extended from
\(\Q(v^{1/D})\) to \(\K_D\).
\end{proposition}

The first identity compares the route ``apply $D_M\otimes\id_M$, then
$\ev_t$'' with the route ``apply $J_{M^*,M}$, then $\ev_1$''.
The second identity makes the analogous comparison for $\qtr$.
In the third and fourth identities, the right-hand route begins with the
one-parameter coevaluation, then applies the inverse tensor comparison, and
finally applies $D_M$ to the dual factor.

\begin{proof}
\emph{Step 1: compute the dual action after transport.}
The antipode formulas \eqref{eq:antipode-roots} imply that the root
generators act on \(\Phi_t(M)^*\) by
\begin{equation}\label{eq:transported-dual-actions}
 E_i^{(t)}\varphi_{-\lambda}
 =t^{-d_i+\langle i,\lambda\rangle}
  E_i^{(1)}\varphi_{-\lambda},
 \qquad
 F_i^{(t)}\varphi_{-\lambda}
 =t^{-\langle\lambda,i\rangle}
  F_i^{(1)}\varphi_{-\lambda}.
\end{equation}
For example, \(E_i^{(t)}\varphi_{-\lambda}\) is supported on weight
\(\lambda-i\); inserting \(S(E_i)=-K_i^{-1}E_i\) and
\eqref{eq:transport-roots} into \eqref{eq:dual-action} gives the first
factor in \eqref{eq:transported-dual-actions}.
The calculation for \(F_i\) uses
\(S(F_i)=-F_i(K_i')^{-1}\).

\emph{Step 2: prove that $D_M$ intertwines the generators.}
On the source \(\Phi_t(M^*)\), formula \eqref{eq:transport-roots} uses the
input weight \(-\lambda\).
Since \(E_i^{(1)}\varphi_{-\lambda}\) has weight \(-(\lambda-i)\), one gets
\begin{align}
 D_M(E_i^{(t)}\varphi_{-\lambda})
 &=
 t^{-\langle\lambda,\lambda\rangle-d_i+\langle i,\lambda\rangle}
 E_i^{(1)}\varphi_{-\lambda}
 =
 E_i^{(t)}D_M(\varphi_{-\lambda}),
 \label{eq:dual-E-check}\\
 D_M(F_i^{(t)}\varphi_{-\lambda})
 &=
 t^{-\langle\lambda,\lambda\rangle-\langle\lambda,i\rangle}
 F_i^{(1)}\varphi_{-\lambda}
 =
 F_i^{(t)}D_M(\varphi_{-\lambda}).
 \label{eq:dual-F-check}
\end{align}
The toral actions also agree because both sides have weight \(-\lambda\).
This proves that \(D_M\) is a module isomorphism.

\emph{Step 3: compare the four cup-cap maps.}
On \(\varphi_{-\lambda}\otimes m_\lambda\), the tensorator
\(J_{M^*,M}\) contributes
\(t^{\langle\lambda,-\lambda\rangle}
=t^{-\langle\lambda,\lambda\rangle}\), exactly the factor contributed by
\(D_M\); this proves \eqref{eq:dual-evaluation-transport}.
The same calculation on
\(m_\lambda\otimes\varphi_{-\lambda}\) proves
\eqref{eq:qtr-transport}, because the factor
\(v_{-\lambda}^2\) in \(\qtr\) is independent of \(t\).
For each homogeneous basis vector of weight \(\lambda\),
\(J^{-1}\) contributes \(t^{\langle\lambda,\lambda\rangle}\) to either
coevaluation formula and \(D_M\) contributes its reciprocal.
Put $q_\lambda=\langle\lambda,\lambda\rangle$.
The four scalar checks can be summarized without suppressing any factor as
\begin{equation}\label{eq:dual-scalar-table}
\begin{array}{c|c|c|c}
 \text{map}&\text{factor from }D_M&\text{factor from }J^{\pm1}&\text{comparison}\\ \hline
 \ev&t^{-q_\lambda}&t^{-q_\lambda}&\text{equal}\\
 \qtr&t^{-q_\lambda}&t^{-q_\lambda}&\text{equal}\\
 \coev&t^{-q_\lambda}&t^{q_\lambda}&1\\
 \coqtr&t^{-q_\lambda}&t^{q_\lambda}&1
\end{array}
\end{equation}
In the last two rows the factors multiply, so they cancel term by term.
This proves \eqref{eq:coev-transport}--\eqref{eq:coqtr-transport} and
completes the proof.
\end{proof}

\subsection{A general ordinary-tangle transport theorem}

% P3.06
We now formulate the formal principle justified by the preceding
computations.
To keep the roles of the hypotheses separate, we introduce four groups of
data in order: a change of coefficients, two tangle functors, comparisons of
their boundary objects, and scalar defects for elementary slices.
The first group is the coefficient change.
Let \(R_0,R_1\) be commutative rings with identity, let
\(\sigma:R_0\to R_1\) be a ring homomorphism, and let
\(\sigma_*\) denote extension of coefficients along \(\sigma\).
A ring homomorphism preserves addition, multiplication, and the identity
element; thus \(\sigma(r+r')=\sigma(r)+\sigma(r')\),
\(\sigma(rr')=\sigma(r)\sigma(r')\), and \(\sigma(1)=1\).
Explicitly, for an \(R_0\)-module \(M\) and an \(R_0\)-linear map
\(f:M\to N\), define
\begin{equation}\label{eq:scalar-extension-definition}
 \sigma_*M=R_1\otimes_{R_0,\sigma}M,
 \qquad
 \sigma_*f=\id_{R_1}\otimes f:\sigma_*M\longrightarrow\sigma_*N,
\end{equation}
where the subscript \((R_0,\sigma)\) means that \(R_1\) is regarded as a
right \(R_0\)-module through \(a\cdot r=a\sigma(r)\).
The second group consists of the two tangle functors to be compared.
For \(a\in\{0,1\}\), let
\begin{equation}\label{eq:tangle-functors}
 \mathcal F_a:\mathsf{OTa}\longrightarrow\mathcal C_a
\end{equation}
be a strict monoidal functor from the ordinary oriented-tangle category to an
\(R_a\)-linear monoidal category.
The term \(R_a\)-linear means that every morphism space is an \(R_a\)-module
and that composition and tensor product of morphisms are bilinear over
\(R_a\).
Assume that the tensor unit of \(\mathcal C_a\) is \(R_a\) and that each of
its endomorphisms is multiplication by a scalar in \(R_a\).
Thus a closed link \(\mathcal L:\varnothing\to\varnothing\) defines a unique
scalar \(I_a(\mathcal L)\) by
\begin{equation}\label{eq:abstract-closed-invariant}
 \mathcal F_a(\mathcal L)
 =I_a(\mathcal L)\id_{R_a}.
\end{equation}

% P3.07
Assume that, after extension along \(\sigma\), the two target categories are
realized in a common \(R_1\)-linear category, so that their objects and
morphisms can be compared by maps in that category.
The third group consists of the comparison maps at the signed boundary
words.
For every sign word \(\mathbf X\), suppose there is an invertible boundary
comparison
\begin{equation}\label{eq:boundary-comparison}
 B_{\mathbf X}:\mathcal F_1(\mathbf X)
 \longrightarrow\sigma_*\mathcal F_0(\mathbf X),
 \qquad B_\varnothing=\id_{R_1}.
\end{equation}
The term boundary refers to the signed endpoints represented by
\(\mathbf X\).
Coherence means that the comparison for a concatenated word is obtained by
iterating the same two-factor comparison and is independent of
parenthesization.

The fourth group consists of the scalar discrepancy allowed at each local
piece of a tangle.
An \emph{elementary slice} is a morphism obtained by placing one of the six
nonidentity generators between any finite numbers of identity strands; thus
it has the form
\(E=\id_{\mathbf U}\otimes G\otimes\id_{\mathbf V}:
\mathbf X\to\mathbf Y\) for sign words \(\mathbf U,\mathbf V\) and a basic
generator \(G\).
For an elementary slice $E:\mathbf X\to\mathbf Y$, the comparison without a
scalar follows the three-arrow route
\begin{equation}\label{eq:local-transport-route}
 \mathcal F_1(\mathbf X)
 \xrightarrow{\ B_{\mathbf X}\ }
 \sigma_*\mathcal F_0(\mathbf X)
 \xrightarrow{\ \sigma_*\mathcal F_0(E)\ }
 \sigma_*\mathcal F_0(\mathbf Y)
 \xrightarrow{\ B_{\mathbf Y}^{-1}\ }
 \mathcal F_1(\mathbf Y).
\end{equation}
The composite in \eqref{eq:local-transport-route} is
$B_{\mathbf Y}^{-1}\sigma_*(\mathcal F_0(E))B_{\mathbf X}$, and
$\alpha(E)$ measures the permitted scalar difference between this composite
and $\mathcal F_1(E)$.
Suppose every elementary slice satisfies
\begin{equation}\label{eq:local-scalar-transport}
 \mathcal F_1(E)
 =
 \alpha(E)B_{\mathbf Y}^{-1}
 \sigma_*\!\left(\mathcal F_0(E)\right)B_{\mathbf X},
 \qquad \alpha(E)\in R_1^\times.
\end{equation}
Here \(R_1^\times\) is the multiplicative group of units of \(R_1\), and
\(\alpha(E)\) is called the local scalar defect.
Assume that the defects respect the defining relations of
\(\mathsf{OTa}\), so their products give a decomposition-independent scalar
\(\alpha(T)\) for every tangle \(T\).

\begin{theorem}[Ordinary-tangle parameter transport]
\label{thm:scalar-transport}
Under the preceding hypotheses, every ordinary oriented tangle
\(T:\mathbf X\to\mathbf Y\) satisfies
\begin{equation}\label{eq:global-scalar-transport}
 \mathcal F_1(T)
 =
 \alpha(T)B_{\mathbf Y}^{-1}
 \sigma_*\!\left(\mathcal F_0(T)\right)B_{\mathbf X}.
\end{equation}
For a closed oriented link,
\begin{equation}\label{eq:closed-scalar-transport}
 I_1(\mathcal L)
 =\alpha(\mathcal L)\sigma\!\left(I_0(\mathcal L)\right).
\end{equation}
If every local defect equals \(1\), the transport is called strict and the
closed invariants are exactly equal after coefficient extension.
\end{theorem}

\begin{proof}
By the presentation recalled in Subsection~2.5, an oriented-tangle diagram
can be cut by horizontal levels so that each layer is an elementary slice or
an identity morphism.
Formula \eqref{eq:local-scalar-transport} holds for each elementary slice by
hypothesis and for an identity morphism with scalar defect one.
We first display this cancellation for two consecutive slices
$E_1:\mathbf X_0\to\mathbf X_1$ and
$E_2:\mathbf X_1\to\mathbf X_2$:
\begin{align}
 \mathcal F_1(E_2\circ E_1)
 &=\mathcal F_1(E_2)\mathcal F_1(E_1)\notag\\
 &=\alpha(E_2)\alpha(E_1)B_{\mathbf X_2}^{-1}
   \sigma_*\mathcal F_0(E_2)
   \underbrace{B_{\mathbf X_1}B_{\mathbf X_1}^{-1}}_{\id}
   \sigma_*\mathcal F_0(E_1)B_{\mathbf X_0}\notag\\
 &=\alpha(E_2)\alpha(E_1)B_{\mathbf X_2}^{-1}
   \sigma_*\mathcal F_0(E_2\circ E_1)B_{\mathbf X_0}.
 \label{eq:two-slice-cancellation}
\end{align}
Thus the comparison maps at the internal boundary $\mathbf X_1$ cancel
exactly.
The local defects multiply, and the relation-compatibility hypothesis makes
their product independent of the chosen cutting and diagram.
Repeating the two-slice calculation gives an induction on the number of
layers and proves \eqref{eq:global-scalar-transport}.
For a closed link, both boundary words are empty and
\(B_\varnothing=\id_{R_1}\); applying
\eqref{eq:abstract-closed-invariant} gives
\eqref{eq:closed-scalar-transport}.
\end{proof}

\subsection{Strict transport of the Fan--Ma--Xing functor}

% P3.08
Let \(M_1\) be a finite-dimensional integrable type-\(1\) simple
highest-weight \(U_{v,1}\)-module of highest weight \(\Lambda\), and put
\(M_t=\Phi_t(M_1)\).
For the one-letter sign words, define
\begin{equation}\label{eq:one-letter-boundary-comparison}
 B_+=\id_{M_t}:M_t\to\Phi_t(M_1),
 \qquad
 B_-=D_{M_1}^{-1}:M_t^*\to\Phi_t(M_1^*).
\end{equation}
For a concatenation \(\mathbf X\mathbf Y\), define recursively
\begin{equation}\label{eq:word-boundary-comparison}
 B_{\mathbf X\mathbf Y}
 =
 J_{\mathcal T_1(\mathbf X),\mathcal T_1(\mathbf Y)}
 (B_{\mathbf X}\otimes B_{\mathbf Y}),
 \qquad B_\varnothing=\id_{\K_D}.
\end{equation}
Proposition~\ref{prop:tensorator} makes this definition independent of
parenthesization.
The maps \(B_{\mathbf X}\) compare the actual tensor products assigned by
\(\mathcal T_t\) with the transport of those assigned by \(\mathcal T_1\).

The curl-normalization scalar is independent of \(t\), because
\begin{equation}\label{eq:curl-independent}
 f_t(\Lambda,\Lambda)=v^{-\Lambda\cdot\Lambda},
 \qquad a_M(t)=a_M(1).
\end{equation}
Indeed, the antisymmetric exponent
\(\langle\Lambda,\Lambda\rangle-\langle\Lambda,\Lambda\rangle\) is zero.
Consequently, Theorem~\ref{thm:checked-R-transport} transports both normalized
crossings in \eqref{eq:curl} with no scalar defect.
Proposition~\ref{prop:dual-comparison} does the same for the four cups and
caps.
Thus
\begin{equation}\label{eq:zero-defect}
 \alpha(E)=1
 \qquad\text{for every elementary slice }E.
\end{equation}

\begin{theorem}[Equality of the two link invariants]
\label{thm:link-equality}
With \(M_1\) and \(M_t\) as above, every ordinary oriented tangle
\(T:\mathbf X\to\mathbf Y\) satisfies
\begin{equation}\label{eq:full-tangle-transport}
 \mathcal T_t(T)
 =B_{\mathbf Y}^{-1}\mathcal T_1(T)B_{\mathbf X},
\end{equation}
where the one-parameter map is understood after extension to \(\K_D\).
In particular, every oriented link \(\mathcal L\) satisfies
\begin{equation}\label{eq:link-equality}
 \boxed{I_{v,t}^{M_t}(\mathcal L)
 =I_{v,1}^{M_1}(\mathcal L).}
\end{equation}
Therefore, for any two oriented links \(\mathcal L_1,\mathcal L_2\),
\begin{equation}\label{eq:fineness-equivalence}
 I_{v,t}^{M_t}(\mathcal L_1)=I_{v,t}^{M_t}(\mathcal L_2)
 \quad\Longleftrightarrow\quad
 I_{v,1}^{M_1}(\mathcal L_1)=I_{v,1}^{M_1}(\mathcal L_2).
\end{equation}
\end{theorem}

\begin{proof}
Equations \eqref{eq:checked-R-transport},
\eqref{eq:inverse-R-transport}, and
\eqref{eq:dual-evaluation-transport}--\eqref{eq:coqtr-transport}, together
with \eqref{eq:curl-independent}, verify
\eqref{eq:local-scalar-transport} with the boundary maps
\eqref{eq:one-letter-boundary-comparison}--\eqref{eq:word-boundary-comparison}
and the zero defects \eqref{eq:zero-defect}.
Theorem~\ref{thm:scalar-transport} gives
\eqref{eq:full-tangle-transport}.
For a closed link, \(\mathbf X=\mathbf Y=\varnothing\), so both boundary maps
are the identity and \eqref{eq:link-equality} follows.
Applying this equality separately to \(\mathcal L_1\) and
\(\mathcal L_2\) proves \eqref{eq:fineness-equivalence}.
\end{proof}

% P3.09
The last equivalence is an equality of functions, not merely a comparison of
selected matrix entries.
It also explains the role of finite dimensionality: besides ensuring that
the quasi-\(R\) action and coevaluation sums are finite on each input, it
makes the closed evaluation an ordinary finite-dimensional trace.
No assertion is made here for affine type or for modules outside the domain
on which the Fan--Ma--Xing tangle functor is defined.

\subsection{Direct verification on a braid closure}

% P3.10
We give a second, completely explicit verification of the cancellation for
closed links.
For an integer \(m\geq1\), define the iterated tensor comparison

\begin{equation}\label{eq:iterated-tensorator}
 \begin{split}
 J_m:\;&M_t^{\otimes_{\Delta_t}m}
 \longrightarrow \Phi_t(M_1^{\otimes_{\Delta_1}m}),\\
 &J_m(m_{\lambda_1}\otimes\cdots\otimes m_{\lambda_m})
 =t^{\sum_{1\leq a<b\leq m}\langle\lambda_b,\lambda_a\rangle}
  m_{\lambda_1}\otimes\cdots\otimes m_{\lambda_m}.
 \end{split}
\end{equation}
Here every \(m_{\lambda_a}\) is a homogeneous vector of weight
\(\lambda_a\), and the superscripts on the tensor products indicate the
coproduct used in their module structures.
Formula \eqref{eq:iterated-tensorator} follows by applying
\eqref{eq:tensorator} recursively from left to right: adjoining the last
factor contributes
\(t^{\langle\lambda_m,\lambda_1+\cdots+\lambda_{m-1}\rangle}\).
The coherence identity \eqref{eq:tensorator-coherence} shows that the same
map is obtained with any other parenthesization.

% P3.11
The braid group \(B_m\) is the group generated by
\(\sigma_1,\ldots,\sigma_{m-1}\), subject to

\begin{equation}\label{eq:braid-group-relations}
 \sigma_i\sigma_{i+1}\sigma_i
 =\sigma_{i+1}\sigma_i\sigma_{i+1},
 \qquad
 \sigma_i\sigma_j=\sigma_j\sigma_i\quad(|i-j|>1).
\end{equation}
The generator \(\sigma_i\) is the positive crossing of the \(i\)th and
\((i+1)\)st strands, and \(\sigma_i^{-1}\) is its negative crossing.
Let

\begin{equation}\label{eq:braid-representations}
 \rho_t:B_m\longrightarrow\operatorname{GL}(M_t^{\otimes m})
\end{equation}
be the representation that assigns \(C_t^+\) to \(\sigma_i\) and \(C_t^-\)
to \(\sigma_i^{-1}\) in tensor positions \(i,i+1\), with identity maps in
all other positions; define \(\rho_1\) analogously.
The notation \(\operatorname{GL}(W)\) means the group of invertible linear
endomorphisms of the vector space \(W\).
The braid relations hold because the normalized checked \(R\)-matrices
satisfy the Yang--Baxter relation and \(C_t^-=(C_t^+)^{-1}\).
Equations \eqref{eq:checked-R-transport},
\eqref{eq:inverse-R-transport}, and \eqref{eq:curl-independent} give the
conjugacy identity

\begin{equation}\label{eq:braid-conjugacy}
 \rho_t(\beta)=J_m^{-1}\rho_1(\beta)J_m
 \qquad(\beta\in B_m).
\end{equation}

Indeed, the identity first holds for every \(\sigma_i^{\pm1}\); multiplying
the identities for consecutive letters in a braid word cancels the adjacent
factors \(J_mJ_m^{-1}\), proving it for the represented braid \(\beta\).

% P3.12
Define the \emph{closure operator} \(G_M\in\End(M_1)\) by

\begin{equation}\label{eq:closure-operator}
 G_Mm_\lambda=v_{-\lambda}^{2}m_\lambda
 \qquad(m_\lambda\in(M_1)_\lambda).
\end{equation}
This name records its role in the closure calculation: the coefficient
\(v_{-\lambda}^{2}\) is exactly the coefficient appearing in the map
\(\qtr\) in \eqref{eq:qtr}.
It is not an additional generator of the tangle category.
If \(A\in\End(M_1^{\otimes m})\), inserting \(m\) copies of
\(\coqtr\) at the bottom and \(m\) copies of \(\qtr\) at the top gives

\begin{equation}\label{eq:closure-trace}
 \operatorname{cl}(A)
 =\Tr_{M_1^{\otimes m}}\!\left(G_M^{\otimes m}A\right).
\end{equation}
Here \(\operatorname{cl}(A)\) denotes the scalar obtained by closing the
input and output of \(A\), and \(\Tr_W(A)\) is the ordinary matrix trace of
an endomorphism \(A\) of a finite-dimensional vector space \(W\).
To verify \eqref{eq:closure-trace}, choose a homogeneous basis
\(b_1,\ldots,b_N\) of \(M_1\), insert
\(\coqtr(1)=\sum_i b_i\otimes b_i^*\), and then apply \(\qtr\) at the
corresponding upper endpoints.
The surviving summands are precisely the diagonal matrix coefficients of
\(G_M^{\otimes m}A\), and their sum is its trace.

% P3.13
The maps \(J_m\) and \(G_M^{\otimes m}\) commute, because both are diagonal
in every homogeneous tensor-product basis.
Let \(\overline\beta\) be the oriented link obtained by joining the upper
endpoint of every strand of \(\beta\) to the corresponding lower endpoint
outside the braid.
Using \eqref{eq:braid-conjugacy}, \eqref{eq:closure-trace}, commutativity of
the two diagonal maps, and the cyclic identity
\(\Tr_W(AB)=\Tr_W(BA)\), we obtain

\begin{align}
 I_{v,t}^{M_t}(\overline\beta)
 &=\Tr\!\left(G_M^{\otimes m}J_m^{-1}
               \rho_1(\beta)J_m\right)\notag\\
 &=\Tr\!\left(J_mG_M^{\otimes m}J_m^{-1}
               \rho_1(\beta)\right)\notag\\
 &=\Tr\!\left(G_M^{\otimes m}\rho_1(\beta)\right)
 =I_{v,1}^{M_1}(\overline\beta).
 \label{eq:braid-closure-cancellation}
\end{align}
The first and last traces are computed on the same underlying finite
dimensional vector space; the transported and one-parameter tensor powers
differ only in their module actions.
Alexander's theorem says that every oriented link is the closure of an
oriented braid \cite{alexander1923}, so
\eqref{eq:braid-closure-cancellation} independently verifies the closed-link
part of Theorem~\ref{thm:link-equality}.
As shown in Figure~\ref{fig:braid-closure}, the two comparison maps lie inside
the closure, and the same cancellation is expressed in trace notation.

\begin{figure}[t]
\centering
\includegraphics[width=.96\textwidth]{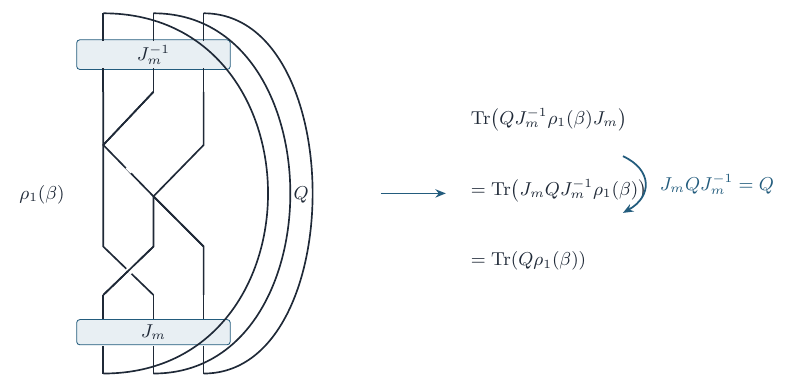}
\caption{Cancellation of the boundary comparison maps in a braid closure.}
\label{fig:braid-closure}
\end{figure}

\subsection{The vector representations of types
\texorpdfstring{$A,B,C,D$}{A, B, C, D}}

% P3.14
We now identify the classical standard modules covered by the theorem.
For type $A_n$, let $\varepsilon_1,\ldots,\varepsilon_{n+1}$ be the standard
orthonormal basis of an $(n+1)$-dimensional Euclidean space; for types
$B_n,C_n,D_n$, let $\varepsilon_1,\ldots,\varepsilon_n$ be the standard
orthonormal basis of an $n$-dimensional Euclidean space.
Write $(\ ,\ )_E$ for this Euclidean inner product.
For a chosen type $X\in\{A_n,B_n,C_n,D_n\}$, the symmetric Cartan form on
the real span of the roots used in \eqref{eq:forms} is normalized by
\begin{equation}\label{eq:ABCD-form-normalization}
 \alpha\mathbin{\cdot}\beta=c_X(\alpha,\beta)_E,
 \qquad c_A=c_C=c_D=1,\qquad c_B=2.
\end{equation}
The factor $c_B=2$ makes the integers
$d_i=\alpha_i\mathbin{\cdot}\alpha_i/2$ relatively prime in type $B_n$,
as required by \eqref{eq:omega}; multiplying the form by this common factor
does not change the fundamental-weight ratios in
\eqref{eq:first-fundamental-weight}.
To specify the two-parameter algebra, fix for each type \(X\) an integral
Ringel matrix \(\Omega^X\) whose symmetrization is this Cartan form.
For example, after ordering the simple roots by their indices, one may take
\begin{equation}\label{eq:ABCD-Ringel-choice}
 \Omega^X_{ii}=d_i,\qquad
 \Omega^X_{ij}=
 \begin{cases}
  \alpha_i\mathbin{\cdot}\alpha_j,&i<j,\\
  0,&i>j.
 \end{cases}
\end{equation}
Indeed, \(\Omega^X_{ij}+\Omega^X_{ji}
=\alpha_i\mathbin{\cdot}\alpha_j\), so \eqref{eq:ABCD-Ringel-choice}
satisfies \eqref{eq:omega} and recovers the displayed symmetric form.
The results below hold for this choice, and more generally for any fixed
Ringel matrix satisfying \eqref{eq:omega} with the same symmetrization.
For type $A_n$, put
\begin{equation}\label{eq:A-weights}
 \overline\varepsilon_i
 =\varepsilon_i-\frac1{n+1}\sum_{k=1}^{n+1}\varepsilon_k
 \qquad(1\leq i\leq n+1).
\end{equation}
For each row indexed by $X$, let $V^X$ denote the vector representation of
Cartan type $X$.
The simple roots and the weights of these vector representations are
\begin{equation}\label{eq:ABCD-roots-weights}
\begin{array}{c|c|c}
X&\text{simple roots}&\text{weights of }V^X\\ \hline
A_n&\alpha_i=\varepsilon_i-\varepsilon_{i+1}
    &(\overline\varepsilon_1,\ldots,\overline\varepsilon_{n+1})\\
B_n&\alpha_i=\varepsilon_i-\varepsilon_{i+1}\ (i<n),\quad
    \alpha_n=\varepsilon_n
    &(\varepsilon_1,\ldots,\varepsilon_n,0,
      -\varepsilon_n,\ldots,-\varepsilon_1)\\
C_n&\alpha_i=\varepsilon_i-\varepsilon_{i+1}\ (i<n),\quad
    \alpha_n=2\varepsilon_n
    &(\varepsilon_1,\ldots,\varepsilon_n,
      -\varepsilon_n,\ldots,-\varepsilon_1)\\
D_n&\alpha_i=\varepsilon_i-\varepsilon_{i+1}\ (i<n),\quad
    \alpha_n=\varepsilon_{n-1}+\varepsilon_n
    &(\varepsilon_1,\ldots,\varepsilon_n,
      -\varepsilon_n,\ldots,-\varepsilon_1)
\end{array}
\end{equation}
In the $D_n$ row, the formula $\alpha_i=\varepsilon_i-
\varepsilon_{i+1}$ is used for $1\leq i\leq n-1$, and the displayed
$\alpha_n$ replaces the usual last difference root.

% P3.15
The smallest admissible lattices containing these vector weights are
\begin{equation}\label{eq:ABCD-lattices}
\begin{array}{c|c|c}
X&Q_X&L_X^{\mathrm{vec}}\\ \hline
A_n&\{a\in\Z^{n+1}:\sum_i a_i=0\}
   &Q_{A_n}+\Z\overline\varepsilon_1\\
B_n&\Z^n&\Z^n\\
C_n&\{a\in\Z^n:\sum_i a_i\equiv0\pmod2\}&\Z^n\\
D_n&\{a\in\Z^n:\sum_i a_i\equiv0\pmod2\}&\Z^n
\end{array}
\end{equation}
The congruence $\sum_i a_i\equiv0\pmod2$ means that the sum of the integer
coordinates is even.
The superscript $\mathrm{vec}$ abbreviates vector representation.
The first lattice in each row is the root lattice, and the last is the
lattice used for the module transport.
Each last lattice contains its root lattice with finite index, so it is
admissible in the sense of \eqref{eq:lattice}.

% P3.16
For each of the four types, let \(\varpi_1\) be the first fundamental weight,
characterized by

\begin{equation}\label{eq:first-fundamental-weight}
 \frac{2\varpi_1\mathbin{\cdot}\alpha_i}
      {\alpha_i\mathbin{\cdot}\alpha_i}=\delta_{1i}
 \qquad(i\in I).
\end{equation}
Thus \(\varpi_1=\overline\varepsilon_1\) in type \(A_n\) and
\(\varpi_1=\varepsilon_1\) in types \(B_n,C_n,D_n\).
A weight \(\Lambda\in L\) is called dominant when
\(2\Lambda\mathbin{\cdot}\alpha_i/
(\alpha_i\mathbin{\cdot}\alpha_i)\in\N\) for every simple root
\(\alpha_i\); hence \(\varpi_1\) is dominant by
\eqref{eq:first-fundamental-weight}.
Let \(V_1^X\) be the usual one-parameter simple highest-weight module of
highest weight \(\varpi_1\), and let \(V_t^X\) be the usual two-parameter
simple highest-weight module of the same highest weight.
These are the vector representations whose weights are listed in
\eqref{eq:ABCD-roots-weights}.
The generic highest-weight construction gives, up to module isomorphism, a
unique integrable simple module of each dominant highest weight
\cite[Section~5]{fan-li2015}.

The transported module \(\Phi_t(V_1^X)\) is simple and has the unchanged
highest-weight vector and highest weight \(\varpi_1\) by
Theorem~\ref{thm:module-transport}.
The preceding uniqueness therefore gives a \(U_{v,t}\)-module isomorphism

\begin{equation}\label{eq:transported-vector-module}
 \psi_X:V_t^X\xrightarrow{\ \cong\ }\Phi_t(V_1^X),
 \qquad X\in\{A_n,B_n,C_n,D_n\}.
\end{equation}
Fix one such isomorphism \(\psi_X\) for each type \(X\).
Changing a color module by an isomorphic module conjugates every boundary
operator by tensor products of \(\psi_X\) and its dual; for a closed link
these conjugating maps cancel.
Consequently, the invariant colored by the usual module \(V_t^X\) equals
the invariant colored by \(\Phi_t(V_1^X)\).

\begin{corollary}[Classical vector representations]
\label{cor:ABCD-vector-equality}
For every oriented link $\mathcal L$, one has
\begin{equation}\label{eq:ABCD-vector-equality}
 \boxed{I_{v,t}^{X,V_t^X}(\mathcal L)
 =I_{v,1}^{X,V_1^X}(\mathcal L)}
\end{equation}
for $X=A_n$ with $n\geq1$, $X=B_n,C_n$ with $n\geq2$, and $X=D_n$ with
$n\geq4$.
The notation records both the Cartan type $X$ and the module used to color
the link.
\end{corollary}

% P3.17
\begin{proof}
The weights in \eqref{eq:ABCD-roots-weights} belong to the corresponding
lattices in the last column of \eqref{eq:ABCD-lattices}.
The Ringel and symmetric forms extend to these lattices over $\Q$, and one
chooses $D$ as in \eqref{eq:denominators}.
The vector modules are finite-dimensional integrable type-$1$ simple
highest-weight modules.
Theorem~\ref{thm:link-equality} gives the equality for
\(\Phi_t(V_1^X)\), and the isomorphism \eqref{eq:transported-vector-module}
transfers it to the usual module \(V_t^X\).
This proves \eqref{eq:ABCD-vector-equality} in all four families.
\end{proof}

% P3.18
The transport theorem also gives all four oriented crossing matrices without
recomputing the quasi-$R$-matrix.
Choose a weight basis $v_1,\ldots,v_N$ of $V_1^X$, where
$N=\dim_{\K_D}V_1^X$, let $v_1^*,\ldots,v_N^*$ be its ordinary dual basis,
and put $\gamma_i=\wt(v_i)$.
The symbol $\dim_{\K_D}V_1^X$ denotes the dimension of $V_1^X$ over
$\K_D$, namely the number of vectors in any basis of this space.
Set $(V_1^X)^+=V_1^X$, $(V_1^X)^-=(V_1^X)^*$, and define
\begin{equation}\label{eq:signed-bases-weights}
 e_{i,1}^+=v_i,
 \qquad e_{i,1}^-=v_i^*,
 \qquad w_i^+=\gamma_i,
 \qquad w_i^-=-\gamma_i.
\end{equation}
The algebraic dual of \(\psi_X\) is the map
\(\psi_X^*:\Phi_t(V_1^X)^*\to(V_t^X)^*\) defined by
\(\psi_X^*(\varphi)=\varphi\circ\psi_X\).
Define module isomorphisms
\begin{equation}\label{eq:signed-vector-identifications}
 \begin{aligned}
 \Psi_X^+&=\psi_X:(V_t^X)^+\longrightarrow\Phi_t((V_1^X)^+),\\
 \Psi_X^-&=D_{V_1^X}^{-1}\circ(\psi_X^*)^{-1}:
 (V_t^X)^-\longrightarrow\Phi_t((V_1^X)^-).
 \end{aligned}
\end{equation}
Thus \(\Psi_X^s\) identifies the usual two-parameter module of sign
\(s\in\{+,-\}\) with the transport of the corresponding one-parameter
module.
On the usual two-parameter modules use the bases
\begin{equation}\label{eq:transported-signed-bases}
 e_{i,t}^+=(\Psi_X^+)^{-1}(v_i)=\psi_X^{-1}(v_i),
 \qquad
 e_{i,t}^-=(\Psi_X^-)^{-1}(v_i^*)
 =\psi_X^*D_{V_1^X}(v_i^*).
\end{equation}
The last expression includes the factor
\(t^{-\langle\gamma_i,\gamma_i\rangle}\) from
\eqref{eq:dual-identification} before the functional is pulled back by
\(\psi_X\).
Thus a $+$ strand carries $V_t^X$, a $-$ strand carries $(V_t^X)^*$, and
both selected bases lie in the stated usual modules.

% P3.19
For signs $s,r\in\{+,-\}$, define the one-parameter matrix coefficients by
\begin{align}
 \widehat R_1^{s,r}
 (e_{i,1}^s\otimes e_{j,1}^r)
 &=\sum_{a,b=1}^N c_{ij}^{ab;s,r}(v)
   e_{a,1}^r\otimes e_{b,1}^s,
 \label{eq:oriented-R1-coefficients}\\
 (\widehat R_1^{s,r})^{-1}
 (e_{i,1}^r\otimes e_{j,1}^s)
 &=\sum_{a,b=1}^N d_{ij}^{ab;s,r}(v)
   e_{a,1}^s\otimes e_{b,1}^r.
 \label{eq:oriented-R1-inverse-coefficients}
\end{align}
The coefficients $c_{ij}^{ab;s,r}(v)$ and $d_{ij}^{ab;s,r}(v)$ are the
unique elements of $\Q(v^{1/D})$ in these two basis expansions.
Let \(\widehat R_{t,\Phi}^{s,r}\) denote the checked $R$-matrix from
\(\Phi_t((V_1^X)^s)\otimes\Phi_t((V_1^X)^r)\) to
\(\Phi_t((V_1^X)^r)\otimes\Phi_t((V_1^X)^s)\).
Each map \(\Psi_X^s\) is a weight-preserving \(U_{v,t}\)-module
isomorphism.
Tensor products of these maps therefore intertwine the actions of every tensor
term of the quasi-\(R\)-matrix on the corresponding modules; they also
intertwine the toral diagonal operators because those operators depend only on
the two weights.
The flip is natural, in the sense that
\[
 P(\Psi_X^s\otimes\Psi_X^r)
 =(\Psi_X^r\otimes\Psi_X^s)P.
\]
Applying these three facts to the definition
\eqref{eq:checked-R-definition} shows that the actual Fan--Ma--Xing checked
$R$-matrix on the usual modules satisfies
\begin{equation}\label{eq:usual-signed-R-identification}
 \widehat R_t^{s,r}
 =\bigl((\Psi_X^r)^{-1}\otimes(\Psi_X^s)^{-1}\bigr)
   \widehat R_{t,\Phi}^{s,r}
   (\Psi_X^s\otimes\Psi_X^r).
\end{equation}
The full two-parameter formulas are
\begin{align}
 \widehat R_t^{s,r}
 (e_{i,t}^s\otimes e_{j,t}^r)
 &=\sum_{a,b=1}^N c_{ij}^{ab;s,r}(v)
 t^{\langle w_j^r,w_i^s\rangle-
    \langle w_b^s,w_a^r\rangle}
 e_{a,t}^r\otimes e_{b,t}^s,
 \label{eq:oriented-Rt-coefficients}\\
 (\widehat R_t^{s,r})^{-1}
 (e_{i,t}^r\otimes e_{j,t}^s)
 &=\sum_{a,b=1}^N d_{ij}^{ab;s,r}(v)
 t^{\langle w_j^s,w_i^r\rangle-
    \langle w_b^r,w_a^s\rangle}
 e_{a,t}^s\otimes e_{b,t}^r.
 \label{eq:oriented-Rt-inverse-coefficients}
\end{align}
Equation \eqref{eq:oriented-Rt-coefficients} has source
$(V_t^X)^s\otimes(V_t^X)^r$ and target
$(V_t^X)^r\otimes(V_t^X)^s$.
Equation \eqref{eq:oriented-Rt-inverse-coefficients} has the reverse source
and target.
Taking $(s,r)=(+,+),\allowbreak(+,-),\allowbreak(-,+),\allowbreak(-,-)$
gives, respectively, the requested
formulas on $V\otimes V$, $V\otimes V^*$, $V^*\otimes V$, and
$V^*\otimes V^*$.

% P3.20
To derive \eqref{eq:oriented-Rt-coefficients}, the input tensorator supplies
$t^{\langle w_j^r,w_i^s\rangle}$ and the inverse output tensorator supplies
$t^{-\langle w_b^s,w_a^r\rangle}$.
Their product is the displayed exponent.
The inverse formula follows in exactly the same way from
\eqref{eq:inverse-R-transport}; its input order is $(r,s)$ and its output
order is $(s,r)$, which explains the different placement of the four
weights.

\subsection{The rank-one vector module}

% P3.21
Consider $U_{v,t}(\mathfrak{sl}_2)$, where $\mathfrak{sl}_2$ is the
rank-one simple Lie algebra of traceless $2\times2$ matrices.
Let $\alpha$ be its simple root, normalize
$\langle\alpha,\alpha\rangle=1$, and put $\omega=\alpha/2$.
Then
$\langle\omega,\alpha\rangle=\langle\alpha,\omega\rangle=1/2$.
Let $V=\Span\{x_+,x_-\}$ with
$\wt(x_+)=\omega$ and $\wt(x_-)=-\omega$.
Here \(\Span\{x_+,x_-\}\) denotes the vector space of all linear
combinations of \(x_+\) and \(x_-\).
Write \(K=K_\alpha\) and \(K'=K_\alpha'\) for the two toral generators.
Their actions, both before and after transport, are
\begin{equation}\label{eq:sl2-toral-action}
 Kx_+=vx_+,\qquad Kx_-=v^{-1}x_-,\qquad
 K'x_+=v^{-1}x_+,\qquad K'x_-=vx_-.
\end{equation}
There is no \(t\)-factor in \eqref{eq:sl2-toral-action} because the rank-one
Ringel form is symmetric.
At $t=1$, choose
\begin{equation}\label{eq:sl2-one-parameter-action}
 Ex_-=x_+,
 \qquad Ex_+=0,
 \qquad Fx_+=x_-,
 \qquad Fx_-=0.
\end{equation}
The transport formulas give
\begin{equation}\label{eq:sl2-transported-action}
 E_tx_-=t^{-1/2}x_+,
 \qquad F_tx_+=t^{1/2}x_-,
 \qquad E_tx_+=F_tx_-=0.
\end{equation}

% P3.22
Because $E_t^2=F_t^2=0$ on $V$, the quasi-$R$ action contains only degrees
$0$ and $\alpha$:
\begin{equation}\label{eq:sl2-quasi-R}
 \Theta_t=1+(v^{-1}-v)F_t\otimes E_t.
\end{equation}
On the only tensor on which the second term is nonzero, the two transport
factors cancel:
\begin{equation}\label{eq:sl2-t-cancellation}
 (F_t\otimes E_t)(x_+\otimes x_-)
 =(t^{1/2}x_-)\otimes(t^{-1/2}x_+)
 =x_-\otimes x_+.
\end{equation}
The Ringel form is symmetric in rank one, so the toral factor is also
independent of $t$.

% P3.23
In the ordered basis
\begin{equation}\label{eq:sl2-basis}
 \mathcal B=(x_+\otimes x_+,x_+\otimes x_-,
 x_-\otimes x_+,x_-\otimes x_-),
\end{equation}
one obtains
\begin{equation}\label{eq:sl2-R-matrix}
 [\widehat R_t]_{\mathcal B}
 =v^{-1/2}
 \begin{pmatrix}
 1&0&0&0\\
 0&0&v&0\\
 0&v&1-v^2&0\\
 0&0&0&1
 \end{pmatrix}
 =[\widehat R_1]_{\mathcal B}.
\end{equation}
The inverse of the middle $2\times2$ block is computed directly, giving
\begin{equation}\label{eq:sl2-R-inverse-matrix}
 [\widehat R_t^{-1}]_{\mathcal B}
 =v^{1/2}
 \begin{pmatrix}
 1&0&0&0\\
 0&1-v^{-2}&v^{-1}&0\\
 0&v^{-1}&0&0\\
 0&0&0&1
 \end{pmatrix}
 =[\widehat R_1^{-1}]_{\mathcal B}.
\end{equation}
Multiplying the matrices in \eqref{eq:sl2-R-matrix} and
\eqref{eq:sl2-R-inverse-matrix} in either order gives the $4\times4$
identity matrix.
As shown in Figure~\ref{fig:rank-one-R}, the matrix is resolved into the
nonzero images of the four ordered basis tensors.

\begin{figure}[t]
\centering
\includegraphics[width=.96\textwidth]{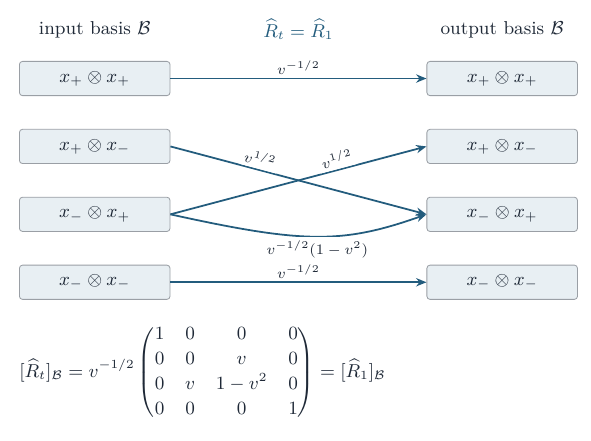}
\caption{The nonzero coefficients of the rank-one checked-$R$ operator.}
\label{fig:rank-one-R}
\end{figure}

\subsection{Sean Clark's theorem as a specialization}
\label{sec:clark}

% P3.24
Sean Clark's quantum covering group uses a quantum variable $q$, a covering
parameter $\pi$ with $\pi^2=1$, and an auxiliary element $\tau$ satisfying
$\tau^2=\pi$ and $\tau^4=1$ \cite[Section~2]{clark2017}.
The root lattice is the integer span of the simple roots, the weight lattice
is the chosen free abelian group of weights containing the root lattice, and
a simple coroot is an element paired integrally with weights and indexed by
the same set as the simple roots.
An enhancer is an integer-valued function on the root and weight lattices
whose congruence conditions record the signs required by parity
\cite[Definition~4.1]{clark2017}.
From an enhancer, Sean Clark constructs a twistor $\mathfrak X$, namely a
semilinear automorphism that changes the quantum variable by
\begin{equation}\label{eq:clark-variable-change}
 \mathfrak X(q)=\tau^{-1}q.
\end{equation}
Here semilinear means that
\(\mathfrak X(f(q)x)=f(\tau^{-1}q)\mathfrak X(x)\) for every coefficient
\(f(q)\) and vector \(x\); the twistor is a parameter-changing operator and
is not a twist generator of \(\mathsf{OTa}\)
\cite[Theorem~4.3]{clark2017}.
On tensor products, the twistor includes an explicit power of $\tau$
depending on the weights and parity data; the resulting weight-diagonal
operator plays the role of the boundary comparison map $B_{\mathbf X}$ in
\eqref{eq:boundary-comparison} \cite[Proposition~4.18]{clark2017}.

% P3.25
Sean Clark's Proposition~4.18 gives the tensor-product comparison, and his
Proposition~4.19 compares the four duality maps
\cite[Propositions~4.18--4.19]{clark2017}.
His Proposition~4.21 compares a crossing, and Proposition~4.22 extends the
crossing and duality comparisons to maps inserted in larger tensor products
\cite[Propositions~4.21--4.22]{clark2017}.
All these comparisons hold up to integral powers of $\tau$
\cite[Propositions~4.21--4.22]{clark2017}.
These powers are the local scalar defects $\alpha(E)$ in
\eqref{eq:local-scalar-transport}, and the writhe normalization in his
Theorem~4.24 supplies the remaining closure factor
\cite[Theorem~4.24]{clark2017}.
The writhe of an oriented diagram is the number of positive crossings minus
the number of negative crossings.

% P3.25a
Fix $n\geq1$, and let $X_{\mathrm{wt}}$ be the weight lattice for the
type-$B_n$ root
datum shared by $\mathfrak{so}(2n+1)$ and the covering form of
$\mathfrak{osp}(1|2n)$.
A root datum consists of a weight lattice, a coweight lattice, chosen simple
roots in the weight lattice, chosen simple coroots in the coweight lattice,
and an integer-valued pairing between the two lattices; pairing the simple
coroots with the simple roots gives the Cartan matrix.
A weight $\lambda\in X_{\mathrm{wt}}$ is called dominant if
$\langle h_i,\lambda\rangle\in\N$ for every simple coroot $h_i$; here
$\langle h_i,\lambda\rangle$ is the integer obtained by pairing the coroot
$h_i$ with the weight $\lambda$.
The simple coroots are characterized by
$\langle h_i,\alpha_j\rangle=a_{ij}$, where $a_{ij}$ are the Cartan
matrix entries of the type-$B_n$ root datum.
Let $\C$ denote the field of complex numbers, and fix
$\iota\in\C$ with $\iota^2=-1$.
Let $q$ be an indeterminate, that is, a formal variable rather than a fixed
number.
Put $R=\C(q)$, the field of rational functions in $q$, and define the field
automorphism
\begin{equation}\label{eq:clark-scalar-map}
 \sigma_\iota:R\longrightarrow R,
 \qquad \sigma_\iota(f(q))=f(\iota^{-1}q).
\end{equation}
For an oriented knot $K$ colored by $\lambda$, write
$J_{\mathfrak{so},K}^{\lambda}(q)$ for Sean Clark's normalized covering
knot invariant at $\tau=1$, and write
$J_{\mathfrak{osp},K}^{\lambda}(q)$ for the corresponding invariant at a
fixed value $\tau=\iota$ with $\iota^2=-1$.
The letter $J$ in these two invariant symbols follows Clark's notation and
is distinct from the boundary comparison map $B_{\mathbf X}$.
The semilinearity in \eqref{eq:clark-variable-change} means precisely that
the twistor becomes $R$-linear after scalar extension along $\sigma_\iota$.
Consequently, Sean Clark's tensor-product, duality, crossing, and writhe
comparisons satisfy the hypotheses of Theorem~\ref{thm:scalar-transport} with
$R_0=R_1=R$, $\sigma=\sigma_\iota$, the tensor-product twistor as
$B_{\mathbf X}$, and the indicated powers of $\iota$ as $\alpha(E)$ for
elementary slices \(E\).

\begin{theorem}[Sean Clark specialization]
\label{thm:clark-specialization}
Let $K$, $\lambda$, $q$, and $\iota$ be as above.
There is an integer $\eta(K,\lambda)$ such that
\begin{equation}\label{eq:clark-comparison}
 J_{\mathfrak{osp},K}^{\lambda}(q)
 =\iota^{\eta(K,\lambda)}
  J_{\mathfrak{so},K}^{\lambda}(\iota^{-1}q).
\end{equation}
This is the specialization of Theorem~\ref{thm:scalar-transport} obtained
from the two covering values $\tau=\iota$ and $\tau=1$.
After the scalar normalization specified by
$\iota^{\eta(K,\lambda)}$ and the displayed change of variable, the two
quantum-covering knot invariants contain the same information.
\end{theorem}

% P3.26
\begin{proof}
Apply Theorem~\ref{thm:scalar-transport} with
$R_0=R_1=R$, $\sigma=\sigma_\iota$, the tensor-product twistor as
$B_{\mathbf X}$, and Sean Clark's local powers of $\iota$ as $\alpha(E)$.
Because all internal comparison maps cancel, the local powers of $\iota$
and the writhe-normalization factor multiply to a single scalar
$\alpha(K)=\iota^{\eta(K,\lambda)}$ for some
$\eta(K,\lambda)\in\Z$.
The closed-link formula \eqref{eq:closed-scalar-transport} therefore gives
\[
 J_{\mathfrak{osp},K}^{\lambda}(q)
 =\iota^{\eta(K,\lambda)}
  \sigma_\iota\!\left(J_{\mathfrak{so},K}^{\lambda}(q)\right).
\]
Using the definition of $\sigma_\iota$ in
\eqref{eq:clark-scalar-map} yields \eqref{eq:clark-comparison}.
This is Sean Clark's Theorem~4.24 in the notation used here.
\end{proof}

% P3.27
The term specialization in Theorem~\ref{thm:clark-specialization} refers to
the general scalar-defect transport theorem and its data
$(\sigma_\iota,B_{\mathbf X},\alpha)$.
For $U_{v,t}$ the defect is identically one by \eqref{eq:zero-defect}; in
Sean Clark's setting it is an integral power of $\tau$.
The two results are therefore the strict and scalar-defect cases of the same
formal cancellation principle.
As shown in Figure~\ref{fig:two-applications}, the figure summarizes these two
applications of the general theorem and displays the resulting invariant
identities.

\begin{figure}[t]
\centering
\includegraphics[width=.98\textwidth]{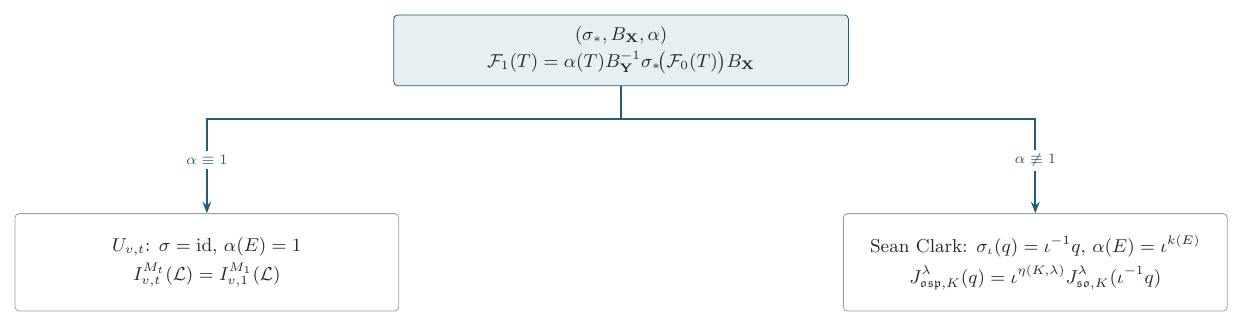}
\caption{The strict and scalar-defect applications of parameter transport.}
\label{fig:two-applications}
\end{figure}

%==============================================================================
\section{Conclusions and Scope}
\label{sec:conclusions}
%==============================================================================

% P4.01
The principal conclusion is the function identity
$I_{v,t}^{M_t}(\mathcal L)=I_{v,1}^{M_1}(\mathcal L)$ for every oriented
link in the stated
finite-type module domain.
It follows from explicit module transport, checked-$R$ conjugacy, strict
transport of the four duality maps, independence of the curl normalization,
and cancellation of the iterated tensorator in the closure trace
\eqref{eq:closure-trace}.
Equation \eqref{eq:fineness-equivalence} shows that equality of values for
two links is equivalent on the two sides, so their distinguishing power is
exactly the same.

% P4.02
The argument is uniform in the finite Cartan type.
It covers every finite-dimensional integrable type-$1$ simple
highest-weight module on an admissible lattice for which the
Fan--Ma--Xing tangle functor is defined.
In particular, Corollary~\ref{cor:ABCD-vector-equality} treats the vector
representations of $A_n,B_n,C_n,D_n$ simultaneously, while
\eqref{eq:oriented-Rt-coefficients} and
\eqref{eq:oriented-Rt-inverse-coefficients} give all four combinations with
the dual space.

% P4.03
The abstract result is Theorem~\ref{thm:scalar-transport}.
The $U_{v,t}$ equality is its strict specialization with $\alpha=1$.
Sean Clark's comparison is its quantum-covering specialization with a
controlled scalar defect, followed by $\tau=\iota$ and $\tau=1$.
This formulation gives the precise sense in which Sean Clark's theorem is
a specialization of the general result proved here.

% P4.04
The scope is finite type, generic parameters, and finite-dimensional
integrable type-$1$ modules equipped with the stated tangle data.
Affine Cartan data, infinite-dimensional modules, roots of unity, and
categories lacking these duality and normalization maps require separate
arguments and are not included in the assertions of this paper.

\bibliographystyle{amsplain}

\end{document}